\documentclass{CSML}

\def\dOi{12(2:1)2016}
\lmcsheading%
{\dOi}
{1--29}
{}
{}
{May\phantom.~12, 2015}
{Apr.~\phantom05, 2016}
{}

\ACMCCS{[{\bf Theory of computation}]: Models of
  computation---Computability---Recursive functions} 

\usepackage{graphicx}
\usepackage{hyperref}
\usepackage{amssymb,amsmath,amsthm}

\DeclareMathOperator{\KP}{\textit{K}\,}
\DeclareMathOperator{\KS}{\textit{C}\,}
\DeclareMathOperator{\BB}{\textit{BB}\,}
\DeclareMathOperator{\m}{\mathbf{m}\,}
\DeclareMathOperator{\dom}{\mathrm{dom}\,}

\newcommand{\cnd}{\mskip2mu|\mskip2mu}
\let\le=\leqslant
\let\ge=\geqslant

\begin{document}

\title[Generic algorithms, halting problem and optimal machines]{Generic algorithms for halting problem and\\ optimal machines revisited\rsuper*}
\author[L.~Bienvenu]{Laurent Bienvenu\rsuper a}
\address{{\lsuper a}LIRMM, CNRS \& University of Montpellier, 161 rue Ada, 34095 Montpellier Cedex 5}
\thanks{{\lsuper a}Supported by the John Templeton Foundation and ANR RaCAF ANR-15-CE40-0016-01 grants}

\author[D.~Desfontaines]{Damien Desfontaines\rsuper b}
\address{{\lsuper b}Google, Brandschenkestrasse 110, 8002 Z\"urich}

\author[A.~Shen]{Alexander Shen\rsuper c}
\address{{\lsuper c}LIRMM, CNRS \& University of Montpellier, 161 rue Ada, 34095 Montpellier Cedex 5; National Research University Higher School of Economics, Moscow}
\thanks{{\lsuper c}Work partially done while visiting the Computer Science Department of the National Research University Higher School of Economics, Moscow. Supported by ANR RaCAF ANR-15-CE40-0016-01 and RFBR 16-01-00362 grants}

\titlecomment{{\lsuper*}An extended abstract of this paper, ``What percentage of programs halt?'', was presented at ICALP 2015 conference~\cite{BienvenuDS2015}.}

\begin{abstract}

The halting problem is undecidable --- but can it be solved for ``most'' inputs? This natural question was considered in a number of papers, in different settings. We revisit their results and show that most of them can be easily proven in a natural framework of optimal machines (considered in algorithmic information theory) using the notion of Kolmogorov complexity.

We also consider some related questions about this framework and about asymptotic properties of the halting problem. In particular, we show that the fraction of terminating programs cannot have a limit, and all limit points are Martin-L\"of random reals. We then consider mass problems of finding an approximate solution of halting problem and probabilistic algorithms for them, proving both positive and negative results.

We consider the fraction of terminating programs that require a long time for termination, and describe this fraction using the busy beaver function. We also consider approximate versions of separation problems, and revisit Schnorr's results about optimal numberings showing how they can be generalized. 

\end{abstract}

\maketitle

\section{Introduction}

One of the most basics theorems of computability theory is that the halting problem is undecidable, i.e.,  there is no algorithm that, given a computation, says whether it terminates or not.  A related result says that for some computations the termination statement is undecidable in G\"odel's sense (neither provable nor refutable). Still, in many cases the termination question is not that hard. It could be that the difficult cases are rare exceptions and that for most cases the answer can be obtained effectively (and perhaps even easily).

This question, while natural, is difficult to formulate. For qualitative questions it does not matter which computational model or programming language we use in the formulation of the halting problem. Technically speaking, all reasonable formulations lead to $m$-complete computably enumerable sets, and all $m$-complete sets are computably isomorphic (Myhill  isomorphism theorem, see, e.g., \cite{Rogers}). However, for quantitative questions the choice of programming language is very important: it is easy to imagine some universal programming language for which most programs terminate (or hang) for some trivial reasons.

One may try to fix some computational model or programming language. For example, we may consider Turing machines with a fixed alphabet, and then ask whether there exists an approximation algorithm for the halting problem, whose success rate among all machines with $n$ states converges to $1$ as $n\to\infty$. This question, however, is sensitive to the details of the definition. For example, it was shown in \cite{HamkinsMiasnikov} that for Turing machines with one-sided tape the success rate may converge to $1$: speaking informally, this happens because most machines fall off the tape rather quickly. This argument, however, does not work for machines with a two-sided tape, for which the similar question remains open. 

Looking for more invariant statements, one should put some restrictions on the way the computations are encoded, and these restrictions could be quite technical. For example, in \cite{CJ1999,KohlerSchindelhauerZiegler} (see also~\cite{Valmari} for a survey of these and some other results), numberings of computable
functions are considered where each function occupies a $\Omega(1)$-fraction of $n$-bit programs for all sufficiently large $n$, together with encodings for pairs that have some special property. However, in~\cite{Lynch1974} a more natural requirement for the programming language (motivated by  the algorithmic information theory) was already suggested. We show (Section~\ref{sec:approximation}) that results about approximate algorithms for halting problem from \cite{CJ1999,KohlerSchindelhauerZiegler}  remain true in this simple setting, and can be easily proved using Kolmogorov complexity. Moreover, they remain true for a weaker requirement than used in \cite{Schnorr1974,Lynch1974}; we discuss this (and some related questions) in Section~\ref{sec:optimality}. 

The next three sections (Sections~\ref{sec:halting-rate}--\ref{sec:busy-beaver}) are devoted to different questions related to the halting problem and its approximate solutions.
In Section~\ref{sec:halting-rate} we consider the fraction of terminating programs among all programs of length at most $n$. We prove that this fraction has no limit as $n\to\infty$, and limit points are Martin-L\"of random reals (even relative to $\mathbf{0}'$). We prove that the $\limsup$ of this fraction is an upper semicomputable $\mathbf{0}'$-random number and every number of this type can appear as $\limsup$ for some machine (Theorem~\ref{th:classification}). In Section~\ref{sec:mass} we consider the task ``find an approximate solution of the halting problem'' as a mass problem in Medvedev's sense, and study whether different versions of this problem can be solved with positive probability by a randomized algorithm, obtaining both positive and negative results for different versions. In Section~\ref{sec:busy-beaver} we consider the following question: how long we need to wait until all terminating programs of size at most $n$, except for a given fraction, do terminate. This question has a natural answer in terms of the busy beaver function, and it can be easily proven using Kolmogorov complexity.

Finally, in the last section (Section~\ref{sec:appendix}) we discuss some generalizations of the previous results.

We assume that the reader has some background in computability theory, algorithmic randomness and Kolmogorov complexity (see, e.g.,~\cite{DowneyH2010,Nies2009,LV2007,VUS2013}). We denote by $\KS$ plain Kolmogorov complexity and by $\KP$ prefix-free Kolmogorov complexity.  

\section{Optimal and effectively optimal machines}\label{sec:optimality} \label{sec:opt-and-eff-opt}

\subsection{Definitions}
The halting problem is described in different ways in different textbooks. Sometimes one considers the diagonal function $\varphi_x(x)$,  i.e., asks whether a program terminates on itself. One can also ask whether a given program terminates on a given input, or whether a given program without input terminates. The latter version looks most suitable for us. Indeed, the diagonal function is considered mostly for historical reasons (Cantor's diagonal argument and  first proofs of undecidability; in these arguments we construct a function that differs from $\varphi_x$ somewhere, and use $x$ as the difference point, but any other sequence of difference points could be used as well). For the second version (when we ask whether $\varphi_x(y)$ is defined) we need to combine $x$ and $y$ into some input pair and measure the size of this pair, if we want to ask about the fraction of correct answers for inputs of given size --- and in this way we get a function of one argument anyway.

So we consider the halting problem for programs without inputs. Then the semantic of the programming language is described by an interpreter machine (algorithm, computable function)~$U$; its inputs and outputs are binary strings. Its inputs are called `programs' and $U(p)$ is the output of program $p$ (undefined if $p$ never terminates).  We use the name \emph{machine} for partial computable functions whose arguments and values are binary strings, and put the following restrictions.

\begin{defi}
A machine $U$ is \emph{universal} if every other machine $V$ is reducible to $U$ in the following sense: there exists a total computable function $h$ such that $V(x)=U(h(x))$ for every $x$ \textup(either both sides are undefined or they are defined and equal\textup).
\end{defi}

Informally, this means that every other programming language $V$ can be effectively translated into $U$ (and $h$ is the translator). In the following definition we additionally require that the length of programs does not increase significantly during the translation. We say that a (total) function $h$ whose arguments and values are binary strings is \emph{length-bounded} if it increases the length at most by a constant, i.e., $|h(x)|\le |x|+c$ for some $c$ and all $x$.

\begin{defi}
A machine $U$ is \emph{effectively optimal} if every other machine $V$ is reducible to $U$ by a total computable length-bounded function $h$.
\end{defi}

\begin{rem}
Effective optimality is an obvious strengthening of optimality (see next definition). It is important that the function $h$ in the definition of an effective optimal machine is total: if we allowed the function~$h$ to be partial computable, the definition would be equivalent to that of optimal machine. Indeed, if $U$ is optimal with some $O(1)$-constant $c$, we can compute $h(x)$ by effectively searching for an $U$-description of $V(x)$ having length at most $|x|+c$.
\end{rem}

A weaker requirement is used in the definition of Kolmogorov complexity where for each machine $U$ we consider the complexity function $\KS_U(x)=\min\{|p|\colon U(p)=x\}$. 

\begin{defi}
A machine $U$ is \emph{optimal} if the complexity function $\KS_U$ is minimal up to $O(1)$ additive term: for every other machine $V$ there exists $c$ such that $\KS_U(x)\le\KS_V(x)+c$ for all~$x$.
\end{defi}

It is easy to see that effectively optimal machines exist: take some universal function $\Phi(p,x)$ such that every machine appears among the functions $\varphi_p(x)=\Phi(p,x)$, and then let $U(\langle p\rangle x)=\Phi(p,x)$ where $\langle p\rangle$ is some self-delimited encoding of $p$ (say, all bits are doubled and $01$ is added at the end). It is also easy to see that every effectively optimal machine is optimal. The function $\KS_U$ for some fixed optimal $U$ is called \emph{Kolmogorov complexity}.

The notion of an optimal machine was introduced and used by Solomonoff and Kolmogorov (see \cite{LV2007} for the historical account); the notion of an effectively optimal machine was introduced in~\cite{Schnorr1974}.  In this paper Schnorr used the name ``optimal enumeration''; we use the name ``effectively optimal'' as a reminder that this is an effectivization of the optimality requirement used by Solomonoff and Kolmogorov.

\begin{rem}\label{rem:effective}
One can also require that the translator function $h$ in the definition of an effectively optimal machine can be found effectively given a machine $V$. This does not give a stronger notion, however: every effectively optimal machine has this property. Indeed, the construction of an effectively optimal machine $U$ given above guarantees this property if $\Phi$ is a G\"odel universal function; then $U$ can be reduced to arbitrary effectively optimal machine $U'$, and this reduction guarantees the same property for $U'$.
\end{rem}

Schnorr noted in~\cite{Schnorr1974} that there exist optimal machines that are not effectively optimal. For example, the machine that keeps the last bit unchanged may be optimal (e.g., one may apply an optimal machine to the preceding bits) but cannot be effectively optimal and even universal. If it were, the translator function could be used to separate the computations that give output~$0$ and output~$1$ (the standard inseparable sets). There are many other examples of optimal but not effectively optimal machines. For example, one may consider an optimal machine where the empty string $\Lambda$ has unique preimage $\Lambda$ (all other programs returning $\Lambda$ are suppressed). An interesting class of optimal machines that cannot be effectively optimal is given in the following proposition.

\begin{defi}
A machine is called \emph{left-total} if for every $n$ the $n$-bit strings in its domain form an initial segment in the lexicographic ordering. 
\end{defi}

\begin{prop}\label{prop:left-total}
The exists a left-total optimal machine but no left-total machine can be effectively optimal.
\end{prop}

\begin{proof}
Every machine $U$ can be converted to a left-total machine $U'$ without changing the complexity function: when a new string~$p$ of length~$n$ appears in the domain of $U$, we add to the domain of $U'$ the lexicographically-least string $q$ of length~$n$ which is not in this domain, and set $U'(q)=U(p)$. To prove that a left-total machine $U$ cannot be effectively optimal, let us assume that $U$ is left-total and effectively optimal and show that in this case for a given enumerable\footnote{For brevity we use the name `enumerable' for recursively (computably) enumerable sets.} set $W$ one can effectively construct a different enumerable set $W'$, thus getting a contradiction with the fixed point theorem for enumerable sets. Indeed, the set $W$ is a domain of some function $V$ that is reducible to $U$, so $V(x)=U(h(x))$ for some length-bounded computable total function $h$. Since $h$ is length-bounded, we can find two strings $u$ and $v$ such that $h(u)$ and $h(v)$ have the same length. If, say, $h(u)\le h(v)$ in the lexicographic ordering, we know that $v\in W$ implies $u\in W$, so the set $W'=\{v\}$ is guaranteed to be different from $W$. Note that we used Remark~\ref{rem:effective} in this argument, since the transformation of $W$ into $W'$ should be effective to get a contradiction with the fixed-point theorem. 
\end{proof}

Here are two other examples of optimal machines that cannot be effectively optimal:

\begin{prop}\hfill
\begin{enumerate}[label=\({\alph*}]

\item There exists an optimal machine whose domain is a simple set \textup(in Post's sense\textup{:} the complement is infinite but does not contain an infinite enumerable set\textup{);} such a machine cannot be effectively optimal.

\item There exists an optimal machine such that $U(p)\ne U(p')$ for
  every two different strings $p,p'$ of the same length; such a
  machine cannot be effectively optimal.
\end{enumerate}
\end{prop}

\begin{proof}\hfill
\begin{enumerate}[label=\({\alph*}]
\item To prove the existence of optimal machines with simple domains, consider first an optimal  machine $U$ whose domain contains at most half of the $n$-bit strings for each~$n$. It can easily be constructed, say, by appending a zero to all arguments. Now we modify this machine by extending its domain, and get a machine $V$ with the same complexity function and simple domain. To make the domain simple, for every enumerable set $W_n$ that contains strings of length greater than $n$, we add one of these strings to the domain of $U$. This is done in the following way. 

We simulate $U$ for all arguments and enumerate all $W_n$ in parallel. When we discover that $U(p)=x$ for some $p$ that is not in the domain of $V$ (yet), we let $V(p)=x$ as well. When some string $x$ of length greater than $n$ appears in $W_n$, we look whether $V$ is already defined on this string. If yes, we do nothing and forget about $W_n$. If not, we add this string to the domain of $V$ with some nonsense value (say, the empty string) and again forget about $W_n$. The only problem arises when $U$ becomes defined on some element that was earlier added to the domain of $V$ with nonsense value. Then, like a person in a concert hall whose seat is already occupied, the value $U(p)$ is assigned to some free seat, i.e., to some unused argument of the same length. In this way the complexity function does not increase. Of course, later the legal owner of this new seat may arrive; then she is sent to some free seat of the same length, etc. Our assumption guarantees that we do not run out of places, since for strings of length $N$ only $W_1,\ldots,W_{N-1}$ could require seating at length $N$, and we have $2^{N-1}$ free seats. This ends the existence proof. It remains to note that the domain of a universal machine is $m$-complete and therefore cannot be simple (see, e.g.,~\cite{Rogers}).

\item We may use the construction from the proof of Proposition~\ref{prop:left-total}  but omit the values that already appeared among the arguments of the same length. To show that a machine $U$ with this property cannot be effectively optimal, consider the function $U'$ defined as follows:
$$
  U'(x)=\begin{cases}
               0^{|x|}, \text{ if $|x|\in P$ };\\
                x, \text{ if $|x| \in Q$};\\
                \text{undefined otherwise.}
            \end{cases}
$$
where $P$ and $Q$ are enumerable inseparable sets of natural numbers
(lengths).  If $h$ is a length-bounded computable total function that
reduces $U'$ to $U$, we can separate $P$ and $Q$ using the following
observation: if $n\in Q$, then $h(x)$ should be different for all $x$
of length $n$, and if $n\in P$, then all $h(x)$ are descriptions of
the same string of length at most $n+O(1)$, so there is at most
$n+O(1)$ different possible values for $h(x)$.\qedhere
\end{enumerate}
\end{proof}

\noindent We finish this section with one last observation. One can consider the following property of a machine $U$ that looks weaker than effective optimality: \emph{for every machine $V$ there exists a total length-bounded computable function $h$ such that $U(h(x))$ is equal to $V(x)$ if $V(x)$ is defined, and $U(h(x))$ may be arbitrary if $V(x)$ is undefined}. In other words, we allow the translation of a non-terminating $V$\!-program to be a terminating $U$\!-program; note that this is enough to conclude that $\KS_U(x)\le \KS_V(x)+O(1)$. In fact, this property is not weaker at all:

\begin{prop}\label{prop:effective}
Every machine $U$ with this property is effectively optimal.
\end{prop}

\begin{proof}
Let $V$ be some other machine for which we want to find a length-bounded translator required by the definition of effective optimality. We know that for every machine $V'$ there exist a ``semi-translator'' of $V'$ to $U$, i.e., a length-bounded total computable function with the property described above. Similarly to Remark~\ref{rem:effective}, we may assume that $h$ can be found effectively given $V'$. Let us consider the following machine $V'$; using the fixed-point theorem, we may assume that $V'$ knows the semi-translator $h$ of $V'$ to $U$:
$$
V'(x)=
\begin{cases}
\text{$V(x)$, if $V(x)$ is defined};\\ 
\text{something different from $U(h(x))$, if $U(h(x))$ is defined};\\
\text{undefined otherwise}.
\end{cases}
$$
This definition is understood as follows: On a given input $x$, $V'$ computes $V(x)$ and $U(h(x))$ in parallel until one of the two computations terminates; then the first or the second line is applied (and we do not care whether the other computation terminates, too).

In fact, the second line is never used, since in this case $V'(x)$ is defined and is different from $U(h(x))$, so $h$ is not a semi-translator. So $V'$ is the same function as~$V$ and $U(h(x))$ is undefined if $V'(x)=V(x)$ is undefined, so $h$ is not only a semi-translator for $V$ but also a translator.
\end{proof}

\subsection{Domains of optimal and effectively optimal machines}

In the sequel we consider the halting problem for the domains of optimal and effectively optimal machines (for most results optimality is enough, but not for all, as we will see).  This motivates the following question: which (enumerable) sets are domains of optimal and effectively optimal machines? 
 
The answer to the first question was given in~\cite{CNSS}; for the reader's convenience we reproduce the proof here.

\begin{thm}[\cite{CNSS}]\label{th:CNSS}
A set $S$ is a domain of an optimal machine if and only if the Kolmogorov complexity of the number of strings of length at most $n$ in $S$ is $n-O(1)$\textup: $$\KS(\#\{x\colon (|x|\le n) \land x\in S\}) =  n-O(1).\eqno(*)$$
\end{thm}

\noindent Note that we count the strings of length \emph{at most} $n$ in $S$, not \emph{exactly} $n$; it is possible, say, that an optimal machine is undefined on all strings of even length.

\begin{proof}
Let us prove first the ``only if'' part. Let $H_n$ be the cardinality of the set in question (for some optimal machine). Then $\KS(H_n\cnd n)\le \KS(H_n)\le n$ with $O(1)$-precision since $H_n\le 2^{n+1}$. To prove the reverse inequality $\KS(H_n)\ge n-O(1)$, and even the stronger inequality $\KS(H_n\cnd n)\ge n-O(1)$, assume that we have a program $q$ that maps $n$ to $H_n$ and is $d$ bits shorter than $n$. We may take an $O(\log d)$-bit self-delimiting description of $d$ and append $q$ to it; the resulting string allows us to reconstruct $d$, then $q$, then $n=|q|+d$, then $H_n=[q](n)$. (Here by $[q](n)$ we denote the output of program $q$ on input $n$.) Then we find all strings of length at most $n$ in the domain of the optimal machine, and a string $z$ that has no description of size at most $n$. This gives $\KS(z)>n$ and $\KS(z)\le O(\log d)+n-d+O(1)$ at the same time, so $d=O(1)$, and $\KS(H_n\cnd n)\ge n-O(1)$, therefore $\KS(H_n)\ge n-O(1)$.

To prove the other implication, assume that $(*)$ is true. Fix some optimal machine $U$. For some $c$ (to be chosen later) consider the following process. We enumerate the domain of $U$ and the set $S$ in parallel, and construct a new machine $V$. When some element $x$ in the domain of~$U$ appears, we suspend the enumeration of the $U$\!-domain and wait until some new string $z$ of length at most $|x|+c$ appears in $S$.  Then we declare $V(z):=U(x)$, and resume the enumeration of the domain of~$U$ (in parallel with the enumeration of $S$). As soon as some other element $x'$ in the $U$\!-domain appears, we repeat this procedure, wait for $z'$ of length at most $|x'|+c$ in $S$, declare $V(z'):=U(x')$, etc. Note that elements in the enumeration of $S$ may appear when nobody is waiting for them (e.g., if they are too long); in this case we make $V(z)$ defined for those elements $z$ immediately. The value is not important, for example, we may declare $V(z)$ to be the empty string in this case. 

The domain of $V$ is exactly $S$: each element of $S$ is either used for some $U$\!-value or wasted, but $V$ is defined on it anyway. If for some $c$ the waiting process is always successful (for every element in the domain of $U$ some suitable $z$ in $S$ is found), the machine $V$ is optimal. It remains to show that for large enough $c$ we get a contradiction with our complexity assumption, if it is not the case.

Assume that some $x$ of length $k$ in the domain of $U$ was not matched. This means that after~$x$ is enumerated in our process, no element of length at most $k+c$ appears in $S$. So the number of elements of length at most $k+c$ in $S$ can be reconstructed if we know $x$ and $c$; note that $k$ is determined by $x$, so we do not need to specify it separately. The pair $(x,c)$ has complexity at most $k+O(\log c)$, so this number has complexity at most $k+O(\log c)$ too, and our assumption gives $k+O(\log c)\ge k+c-O(1)$. This is possible only if $c$ is small, and this finishes the proof.
\end{proof}

We now characterize the domains of effectively optimal machines. 

\begin{defi}\label{def:m-optimal}
An enumerable set $S$ is called \emph{$m$-optimal} if every enumerable set~$W$ can be reduced to $S$ by some length-bounded total computable function $h$: 
$$
x\in W \Leftrightarrow h(x)\in S.
$$
\end{defi}\medskip

\noindent Obviously, a domain of an effectively optimal machine is an $m$-optimal set, and, as we will prove soon (Corollary~\ref{cor:effective-optimal-domain}), the reverse is also true.

If we omit the words ``length-bounded'' in Definition~\ref{def:m-optimal}, we get the classical notion of a \emph{$m$-complete} set. All $m$-complete sets are computably isomorphic (Myhill theorem, see, e.g.,~\cite{Rogers}):  if $S$ and $S'$ are two $m$-complete sets, there exists a computable bijection that maps $S$ to $S'$. Using the arguments of Schnorr~\cite{Schnorr1974},  we can prove the following result:\footnote{Schnorr did not consider $m$-optimal sets. He proved in~\cite{Schnorr1974} that any two effectively optimal machines are reducible to each other by a computable bijection that is length-bounded in both directions. We can however use essentially the same argument to prove the corresponding result for $m$-optimal sets, so we use the name ``Myhill--Schnorr theorem'' for Theorem~\ref{th:myhill-schnorr}.}

\begin{thm}[``Myhill--Schnorr theorem'']\label{th:myhill-schnorr}
Every two $m$-optimal sets $S$ and $S'$ are length-bounded isomorphic: there exists a computable bijection $h$ that maps $S$ to $S'$ such that both $h$ and $h^{-1}$ are length-bounded \textup(so $|h(x)|=|x|+O(1)$\textup).
\end{thm}

\begin{cor}\label{cor:effective-optimal-domain}
Every $m$-optimal set is a domain of an effectively optimal machine.
\end{cor}
Indeed, a computable isomorphism between sets $S$ and $S'$ that is length-bounded in both directions, can be used to convert an effectively optimal machine with domain $S$ into an effectively optimal machine with domain $S'$.

\begin{proof}[Proof of Myhill--Schnorr theorem]
The proof goes in two steps. First (Lemma~\ref{lem:schnorr-injection}) we prove that the reduction in the definition of an $m$-optimal set $S$ can be made injective. Then we use some general combinatorial statement (Lemma~\ref{lem:schnorr}) to get a bidirectional length-bounded bijection from two length-bounded injections.

It is convenient to identify strings (in length-lexicographic ordering) with natural numbers; then length-bounded functions become functions $f\colon \mathbb{N}\to\mathbb{N}$ such that $h(n)=O(n)$.

\begin{lem}\label{lem:schnorr-injection}
Let $S$ be an $m$-optimal set, and let $W$ be an arbitrary enumerable set. Then there exist a computable length-bounded injective total function $h$ such that $x\in W \Leftrightarrow h(x)\in S$ for every $x$.
\end{lem}

\begin{proof}[Proof of Lemma~\ref{lem:schnorr-injection}]
As before (see Remark~\ref{rem:effective}) we may assume that a length-bounded total computable reduction of an arbitrary enumerable set $W$ to $S$ can be found effectively. Also, as in the proof of Proposition~\ref{prop:effective}, we use the fixed-point theorem and construct an enumerable set~$W'$ assuming that a length-bounded total computable reduction~$h$ of~$W'$ to~$S$ is known. The set $W'$ is defined as follows: to determine whether some number $i$ belongs to $W'$, we check first whether $h(i)$ appears among $h(0), h(1),\ldots,h(i-1)$. If yes, then $i$ does not belong to $W'$. If not, we start two processes in parallel: we enumerate $W$ waiting until $i$ appears in $W$, and we compute $h(i+1), h(i+2),\ldots$ until $h(i)$ appears in this list. If one of these two events happens, we include $i$ into $W'$ (and $i\notin W'$ otherwise).

Let us show that in fact $W'=W$ and $h$ is injective (so the statement of the lemma is true).  Assume that $h(k)=h(l)$ for $k<l$. Then $l\notin W'$ since $h(l)$ appears in the list of previous values of $h$, and $k\in W'$, since $l$ appears in the list of subsequent values of $h$. But this cannot happen, since $h$ reduces $W'$ to $S$ and $h(k)=h(l)$ implies $(k\in W')\Leftrightarrow(l\in W')$. Now we know that $h$ is injective, and then the definition of $W'$ guarantees that $W'=W$. Lemma~\ref{lem:schnorr-injection} is proven.
\end{proof}

Now we switch to the combinatorial part. 

\begin{lem}\label{lem:schnorr}
Consider two total injective functions $f\colon\mathbb{N}_1\to\mathbb{N}_2$ and $g\colon\mathbb{N}_2\to\mathbb{N}_1$ (here $\mathbb{N}_1=\mathbb{N}_2=\mathbb{N}$, and subscripts $1$ and $2$ are used just to show that we consider two copies $\mathbb{N}_1$ on the left and $\mathbb{N}_2$ on the right). Consider the bipartite graph $E\subset \mathbb{N}_1\times\mathbb{N}_2$ that combines edges of type $(x,f(x))$ and of type $(g(y),y)$. There exists a bijection $h\colon\mathbb{N}_1\to\mathbb{N}_2$ with the following properties:
\begin{itemize}
\item the vertices connected by $h$ \textup(i.e., $x$ and $h(x)$\textup) belong to the same connected component of~$E$\textup;
\item $h(i)$ is bounded by $\max\{f(i')\mid i'\le i\}$\textup;
\item $h^{-1}(j)$ is bounded by $\max\{g(j')\mid j'\le j\}$\textup;
\end{itemize}
If $f$ and $g$ are computable, $h$ can be made computable.
\end{lem}

(In our application the functions $f$ and $g$ are reductions between two enumerable sets going in opposite directions, so either all the elements of the connected component belong to the corresponding sets, or all elements of the connected component do not belong to the corresponding sets.)

\begin{proof}[Proof of Lemma~\ref{lem:schnorr}]
We construct $h$ step by step: at each stage we have a finite one-to-one correspondence between two finite subsets of $\mathbb{N}_1$ and $\mathbb{N}_2$. Then a new pair is added to the current $h$, and at the same time we modify $f$ and $g$ in such a way that $\hat f$ and $\hat g$ (the current versions of $f$ and $g$) have the following properties:
\begin{itemize}
\item  $\hat f$ and $\hat g$ are injective;
\item the current $h$ is contained both in $\hat f$ and in $\hat g$, meaning that for every pair $u$ in the current domain of $h$, $h(u)=\hat f(u)$ and for every $v$ in the current co-domain of~$h$, $h^{-1}(v)=\hat g(v)$;
\item $\hat f$ and $\hat g$ connect vertices in the same connected component (in the initial graph), so the connected components of the new graph (for $\hat f$ and $\hat g$) are parts of the connected components of the initial graph;
\item $\hat f(i)$ is bounded by $\max\{f(i')\mid i'\le i\}$ (for the original $f$);
\item $\hat g(j)$ is bounded by $\max\{g(j')\mid j'\le j\}$ (for the original $g$).
\end{itemize}
The last two requirements may be reformulated as follows:  changes in $f$ and $g$ may only decrease the quantities $\max (f(1),\ldots,f(i))$ and $\max(g(1),\ldots,g(j))$ for all $i,j$.  So it is enough to check the non-increase of these two quantities for each change in $\hat f$ and $\hat g$.

At each step we take the \emph{minimal} element on one of the sides that is not covered by current~$h$, and add to $h$ some pair that involves this element. Doing this alternatingly for both sides, we guarantee that the final $h$ is a bijection. How is this done? For example, let $u$ be the minimal left element not covered by current $h$. It is currently connected to some $v=\hat f(u)$.  Then (1)~\emph{$u$ and $v$ are in the same component}, and (2)~\emph{$v$ is not covered by $h$}. Indeed, if one of the endpoints of an $\hat f$- or $\hat g$-edge is covered by $h$, then the other one is covered by $h$, too (since $h$ is included in $\hat f$ and $\hat g$, and both $\hat f$ and $\hat g$ are injective).

\begin{description}
\item[Case 1] $\hat g(v)=u$. This case is simple: we add $(u,v)$ to $h$, leaving $\hat f$ and $\hat g$ unchanged.\

\item[Case 2]: $w=\hat g(v)\ne u$. We would like to let $\hat g(v)$ to be $u$, but need to be careful and consider several cases. Note first that $w>u$ since $w$ is not covered by $h$ and all elements smaller than $u$ are covered ($u$ is minimal).

\begin{figure}[h]
\begin{center}
\includegraphics{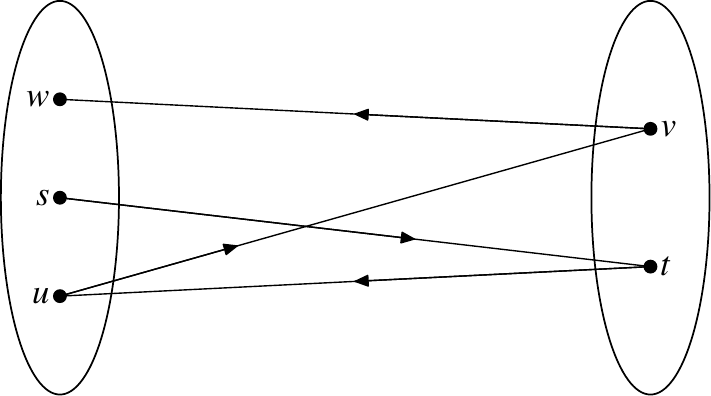}
\end{center}
\end{figure}

\item[Case 2.1] $u$ is not in the image of $\hat g$. Then we let $\hat g(v):=u$, and $g$ is still an injection.  We have decreased some value of $\hat g$ (recall that $w=\hat g(v)$ was greater than $u$), so $\max(\hat g(1),\ldots,\hat g(j))$ could only decrease. Then we proceed as in case 1, i.e., add $(u,v)$ to $h$.

\item[Case 2.2] $u=\hat g(t)$ for some $t$.  Note that $t\ne v$ by injectivity of $\hat g$, and we have four vertices $u,v,w,t$ that are in the same component; none of which is covered by $h$ (for the reasons explained above). We would like to exchange the values of $\hat g(t)=u$ and $\hat g(v)=w$ and let $\hat g(t):=w$, $\hat g(v):=u$, coming again to case 1. The exchange keeps $\hat g$ injective and exchanged values are in the same component. So the only problem is that $\max(\hat g(1),\ldots,\hat g(j))$ should not increase.

\item[Case 2.2.1] $t>v$. In this case we exchange the values that were in the reversed order ($t>v$ but $\hat g(t)=u<w=\hat g(v)$). Such an exchange can only decrease the quantity in question, so we come to case 1 after the exchange and add $(u,v)$ to $h$.

\item[Case 2.2.2] $t<v$. In this case we give up changing $\hat g$ and change $\hat f$ instead. 

\item[Case 2.2.2.1] $t$ is not in the image of $\hat f$. Then we let $\hat f(u):=t$ as in case 2.1, decreasing the value of $\hat f(u)$, and add $(u,t)$ to $h$.

\item[Case 2.2.2.2] $t=\hat f(s)$ for some $s$. Then $s\ne u$ and is not covered by $h$,  therefore $s>u$ since $u$ was minimal element not covered by $h$. All elements $s,t,u,v$ are in the same component, are not covered by $h$, and the values $v=\hat f(u)$ and $t=\hat f(s)$ go in the reversed order, so we exchange them by letting $\hat f(u):=t$ and $\hat f(s):=v$, as in the case 2.2.1.
\end{description}

\noindent This finishes the case analysis.

This construction is not effective: why do we get a computable $h$ if we cannot distinguish between cases where we need to make an exchange or not? This is not a problem, however. Note that we can check whether there exists some $t<v$ such that $g(t)=u$. Depending on this, we either perform the exchange for the values of $g$, or for the values of $f$. We do not know whether the exchange is really needed, but still can effectively write a program that takes into account this exchange when and if it turns out to be necessary (the necessity of exchange becomes obvious when we come to the stage when the corresponding value needs to be computed).
\end{proof}
By Lemma~\ref{lem:schnorr-injection} and Lemma~\ref{lem:schnorr}, the Myhill--Schnorr theorem is proven.
\end{proof}

So we have studied two conditions for a machine: optimality and effective optimality. The second is stronger; it makes the domain of the machine $m$-optimal and therefore determines it uniquely up to a computable length-preserving (up to $O(1)$ additive term) bijection. For most of the results in the next sections we will only need optimality, and use the criterion provided by Theorem~\ref{th:CNSS}.

\section{Approximate algorithms for halting problem}\label{sec:approximation}

Now we consider the halting problem as the decision problem for the domain of some fixed optimal machine $U$. Theorem~\ref{th:CNSS} guarantees that this problem is undecidable (otherwise the complexity $\KS(H_n)$ would be logarithmic in $n$). So there is no algorithm that correctly answers all questions whether $x$ belongs to the domain of $U$ or not.

For every algorithm $A$ (in this section we understand ``algorithm" as partial computable function) we consider the set of inputs where $A$ errs (as a decision algorithm for halting problem), i.e., the set of strings $x$ such that $x\in \dom(U)$ and $A(x)$ does not output $1$, or $x\notin \dom(U)$ and $A(x)$ does not output $0$. Note that there are no restrictions for $A$ yet. Later we will consider algorithms that are total, but may give incorrect answers (\emph{coarse} computation), and also non-total algorithms that never give wrong answers (\emph{generic} computation).

For each $n$ we consider the fraction $\varepsilon_n(A)$ of errors among all strings of length at most $n$ (counting both places where $A$ is undefined, and places where it produces a wrong answer). We are interested in the asymptotic behavior of $\varepsilon_n(A)$ as $n\to\infty$ (since $A$ always can be adjusted on a finite set).  The sequence $\varepsilon_n$ may not converge, so we consider two quantities  $\overline{E}(A)=\limsup \varepsilon_n(A)$ and $\underline{E}(A)=\liminf\varepsilon_n(A)$. In terms of these quantities, we may describe the behavior of the error rate as follows:

\begin{thm}\label{th:main}\hfill
\begin{enumerate}[label=\({\alph*}]

\item
$\underline{E}(A)>0$ for every $A$;

\item
there exists $\varepsilon>0$ such that $\overline{E}(A)\ge \varepsilon$ for every $A$;

\item
for every $\varepsilon>0$ there exists $A$ such that $\underline{E}(A)\le\varepsilon$;

\item
if we consider only algorithms that may be undefined but never produce wrong answers,
then the statement ``for every $\varepsilon>0$ there exists an algorithm $A$ of this type such that $\underline{E}(A)\le\varepsilon$'' may be true or false depending on the choice of the universal machine~$U$\textup; for example, it always holds when~$U$ is left-total and never holds when~$U$ is effectively optimal. 
\end{enumerate}
\end{thm}

\noindent These results (in a different setting and with additional restrictions) were proved in \cite{Lynch1974,CJ1999,KohlerSchindelhauerZiegler}; we show that they can be easily proven using Kolmogorov complexity.

\begin{proof}
(a)~Knowing $n$ and some upper bound $2^{n-d}$ for the number of errors (of both types: $A$ is undefined or the value is wrong) that $A$ makes for strings of length at most $n$, we wait until~$A$ becomes defined on all strings of those lengths except for $2^{n-d}$ ones. Then we count the number of positive answers; it differs by at most $O(2^{n-d})$ from $H_n$, the number of strings of length at most $n$ in the domain of $U$. The difference can be specified by $n-d+O(1)$ bits, so the complexity  $\KS(H_n\cnd n)$ is bounded by $\KP(d\cnd n)+(n-d)+O(1)$, where $\KP(d\cnd 
n)$ in the prefix complexity of $d$ given $n$, the minimal length of self-delimiting program that maps $n$ to $d$. Note that the $O(1)$-constant depends on $A$. The bound $\KS(H_n\cnd n)\ge n-O(1)$ (see Theorem~\ref{th:CNSS} and its proof) then implies that $d-\KP(d\cnd n)\le O(1)$, so $d=O(1)$. This provides the required bound $2^{-O(1)}$ for the fraction of errors for all large enough $n$.

(b)~The difference with (a) is that now the threshold for the number of errors should be the same for all approximation algorithms (while in (a) each algorithm has its own threshold). On the other hand, we only need to show that the fraction of errors is above the threshold infinitely often (and not necessarily for all large enough $n$). So the same argument with small modification works. It can be used with $A$ as a parameter to prove that 
$$
\KS(H_n\cnd n)\le \KP(d\cnd n)+(n-d)+\KP(A\cnd n)+O(1),
$$
if the number of errors is bounded by $2^{n-d}$, and now $O(1)$ does not depend on $A$. It remains to note that for some $c$ and every $A$ there exist infinitely many $n$ such that $\KP(A\cnd n)\le c$ (consider $n$ whose binary representation starts with a self-delimiting program for $A$). We now get (in the same way as before) the uniform lower bound for the number of errors that arbitrary algorithm $A$ makes (but only for $n$ that make $A$ simple). 

(c)~For each $n$ consider $\rho_n$, the fraction of strings of length at most $n$ that belong to the domain of $U$. We will see later that $\rho_n$ does not converge, but we may still consider $\overline\rho=\limsup_n \rho_n$. Now consider a rational number $r$ that is smaller than $\overline\rho$ but is very close to it. There are infinitely many lengths for which the fraction of terminating computations exceeds $r$, and these lengths can be discovered ultimately. We may find a computable increasing sequence of lengths having this property. Consider some length $n$ in this sequence. Imagine that we run $U$ for all inputs of length at most $n$ until we get a fraction~$r$ of terminating computations (i.e., corresponding inputs for $U$). This will happen at some point (according to the construction of the sequence).  If we give positive answers for these inputs and negative answers for all others, the error rate is small if $n$ is sufficiently large. Indeed, since $\limsup \rho_n =\rho$, the value of $\rho_n$ cannot exceed $\overline\rho$ significantly, and $r$ is close to $\overline\rho$, so we have found all positive answers except for a small fraction. This works for each large enough $n$ from the sequence, but different $n$ can lead to different answers for the same input. Let us assume that the lengths in the sequence increase fast enough (this can be done without loss of generality) and agree that we use the minimal length in the sequence that covers our input. If $n_i$ and $n_{i+1}$ are neighbor elements of the sequence, then $n_{i+1}$ is used for all strings of length between $n_i$ and $n_{i+1}$, and strings of length at most $n_i$ form a negligible fraction among strings of length at most $n_{i+1}$.

(d) The claim here consists of two parts --- positive and negative. For the positive part we need to show that for a left-total machine one can construct an algorithm that never provides incorrect answers and for infinitely many $n$ terminates on most strings of length at most $n$ (to be exact, on more than an $(1-\varepsilon)$-fraction of them). Let us explain informally why left-totality helps. If we know how many strings of given length $n$ appear in the left-total set, we know which strings are there (initial segment of this length). If we have some lower and upper bounds for this number that differ by some small~$k$, we can provide correct answers for all strings except for $k$ strings in the middle. However, this argument works only for strings of length \emph{exactly} $n$, and not for strings of length \emph{at most} $n$, so arguments used to prove~(c) do not work now and should be changed.

For each $n$ let us consider the fraction $\tau_n$ of strings of length \emph{exactly} $n$ that belong to $\dom(U)$. As before, there exists some $\tau=\limsup \tau_n$. Consider two rational numbers $r<\tau$ and $r'>\tau$ close to each other. According to the definition of $\limsup$, we have $\tau_n<r'$ for all $n$ greater than some $N$; there are infinitely many $n>N$ such that $\tau_n>r$. The values of $n$ such that $\tau_n>r$ can be enumerated. Therefore, there exists an infinite computable increasing sequence $n_1<n_2<n_3<\ldots$ of integers such that for each $i$ the fraction $\tau_{n_i}$ of $n_i$-bit strings that belong to $\dom(U)$ is between $r$ and $r'$. The difference between $r$ and $r'$ can be made arbitrarily small (though the values of $r$, $r'$, and $N$ cannot be found effectively; they just exist). So we can give correct answers for strings of length $n_i$ without errors and with few omissions, saying ``yes'' for the bottom $r$-fraction and ``no'' for top $(1-r')$-fraction.

This is not enough for us, since in our definition the fraction of prediction failures is calculated in the set of all strings of length \emph{at most $n$}, and even if we know everything for $n$-bit strings, this covers only half of the strings in question. So in this way we cannot make the error less than $1/2$.

But we can repeat the trick: consider the lengths $n_1-1, n_2-1, n_3-1,\ldots$ and the fractions $\tau_{n_i-1}$ of strings of these lengths that belong to $\dom(U)$. The sequence $\tau_{n_i-1}$ again has some $\limsup$, and one can find rational numbers $s<s'$ that are close to each other, and some computable subsequence $m_1<m_2<m_3<\ldots$ of the sequence $n_i$ such that $\tau_{m_i-1}$ is always between $s$ and $s'$. In this way we can provide correct answers for almost all strings of two subsequent lengths $m_{i}-1$ and $m_i$, thus reducing the error from $1/2$ to $1/4$ (approximately).

This construction can be repeated once more, giving error rate close to $1/8$ (for lengths in some computable subsequence), etc. Repeating this trick finitely many times, we get arbitrarily small error. Note that for this we need only finitely many bits of advice, so this still gives a partial computable predictor. The statement (d) is proven.
\smallskip

For the negative part of (d), when~$U$ is effectively optimal, we follow Lynch~\cite{Lynch1974}:
\begin{lem}\label{lem:sparse-simple}
There exists a sparse simple set $S$.
\end{lem}
Recall that a simple set is an enumerable set whose complement is infinite but does not contain an infinite enumerable subset; sparsity means that the fraction of strings of length $n$ that belong to $S$ tends to $0$ as $n\to\infty$.
\begin{proof}[Proof of Lemma~\ref{lem:sparse-simple}]
Let $S$ be the set of highly compressible strings, i.e., strings $x$ such that $\KS(x)< |x|/2$. There is at most $O(2^{n/2})$ compressible strings of length $n$, so $S$ is sparse. If the complement of $S$ contains an infinite enumerable subset, then for every $N$ we can enumerate this set until we get a string of length greater than $2N$ and complexity greater than $N$,  and this string has complexity only $O(\log N)$ --- a contradiction.
\end{proof}

An algorithm that tries to decide a sparse simple set $S$ and is not allowed to give wrong answers can be defined only on a set of density $0$: it can give positive answers only for elements of $S$, and they are rare; on the other hand, it can give only finitely many negative answers since $S$ is simple. 

As we have seen, every enumerable set is reducible to the $m$-optimal set $\dom(U)$ by a length-bounded total computable injection. Consider such a reduction for the sparse simple set~$S$. This reduction can be used to convert an approximate algorithm for $\dom(U)$ without errors to an approximate algorithm for $S$ without errors, and the latter algorithm is undefined almost everywhere. Since the reduction is injective and length-bounded, the image has positive lower density, and it becomes a lower bound for $\underline E(A)$ for every algorithm $A$ that does not make errors.
\end{proof}

Our definition of error rate considers all strings of length at most $n$ to have equal importance. Thus it can be interpreted in a probabilistic fashion: if $A$ is an approximate algorithm, the probability that $A$ makes an error on a $p$ chosen uniformly among strings of length~$n$ is ${>\varepsilon}$ for some $\varepsilon>0$ independent on~$n$. It is natural to combine this implicit randomness with explicit randomness in the algorithm, i.e., allow~$A$ to be a probabilistic algorithm itself. For a randomized algorithm $A$ (that uses fair coin tossing) we define $\varepsilon_n(A)$ as the probability of error on a random input uniformly distributed in the set of all strings of length at most $n$. (We assume, as usual, that the random bits used in the computation are independent of the random choice of an input string.) Then $\overline{E}(A)$ and $\underline{E}(A)$ are defined in the same way as before.  We show now that the negative results of Theorem~\ref{th:main} (a) and (b) remain valid for this setting; note that the algorithms without errors mentioned in (d) do not make much sense in this probabilistic setting (if a probabilistic algorithm never gives wrong answers, it can be replaced by a deterministic search over all possible values of random bits).

\begin{thm}\label{th:main-prob}
For probabilistic algorithms we still have\textup:
\textup{(a)}
$\underline{E}(A)>0$ for every $A$\textup;
\textup{(b)}
there exists $\varepsilon>0$ such that $\overline{E}(A)\ge \varepsilon$ for every $A$\textup.
\end{thm}

\begin{proof}
Knowing $n$ and some upper bound $2^{-d}$ for the probability of error on strings of length at most $n$, we may emulate the behavior of $A$ on all inputs and for all combinations of random bits until the probability not to get an answer (correct or incorrect) goes below $2^{-d}$. In other words, for every string $x$ of length at most $n$ we have $1=p(x)+n(x)+u(x)$, where $p(x)$, $n(x)$, and $u(x)$ are probabilities of positive answer, negative answer, and no answer respectively. The simulation provides better and better lower bounds for $p$ and $n$, and upper bounds for $u$, and we wait until the average of the upper bounds for $u(x)$ (taken over all strings of length at most~$n$) goes below $2^{-d}$. (This must happen since we assumed the probability of error to be  $<2^{-d}$.) 

After that we note that $\sum_x p(x)$ (sum of current approximations taken over all strings of length at most $n$) provides the approximation for the number of strings of length at most $n$ that belong to $\dom(U)$ (assuming that the error probability is indeed bounded by $2^{-d}$). One can also use $\sum (1-n(x))$; both quantities are $O(2^{n-d})$-close to this number. Indeed, let us consider the first quantity: $H_n -\sum_x p(x)$ can be written as $\sum_x (\chi_{\dom(U)}(x)-p(x))$, where $\chi_{\dom(U)}$ is the indicator function for $\dom(U)$, and the sum is taken over all $x$ of length at most $n$. The latter sum can be bounded by $\sum_x |\chi_{\dom(U)}(x)-p(x)|$. Each summand is bounded by a probability of failure for specific $x$: for $x\in\dom(U)$ it is exactly the probability of failure, for $x\notin\dom (U)$ it is smaller (since the algorithm may produce no answer for this $x$). We know by assumption that the average value of the probability of failure is less that $2^{-d}$, thus the sum $\sum_x |\chi_{\dom(U)}(x)-p(x)|$ is bounded by $2^{n-d}$. 

Therefore, we get the upper bound
$$
n\le C(H_n\cnd n)\le n-d+O(\KP(d\cnd n))=n-(d-O(\log d))
$$ 
(we omit $O(1)$-terms), therefore $d-O(\log d)=O(1)$ and $d=O(1)$ as required by~(a).
Here the $O(1)$ term depend on $A$; if $A$ is not fixed, we get the inequality
$$
d-O(\log d)\le \KP(A\cnd n)+O(1),
$$
which gives (b) since $\KP(A\cnd n)\le O(1)$ for infinitely many $n$.
\end{proof}
	
\section{Halting rate and algorithmic randomness}\label{sec:halting-rate}

In this section we consider a different question: instead of deciding the halting problem we study just the fraction of inputs where the optimal machine is defined, and its asymptotic behavior. We consider some optimal machine $U$. Let $H_n$ be the number of inputs of length at most~$n$ on which $U$ halts, and let $\rho^U_n$ be the \emph{fraction} of inputs of length at most~$n$ on which $U$ halts. More precisely, we define $\rho^U_n = H_n / 2^{n+1}$; technically we should put $2^{n+1}-1$ in the denominator, but we ignore this small difference to simplify the notation. As we are only interested in the asymptotic behavior of $\rho^U_n$ this makes no difference. We fix some optimal machine $U$ and write $\rho_n$ instead of $\rho^U_n$ when this machine $U$ is considered.

\begin{thm}\label{th:classification}\hfill
\begin{enumerate}[label=\({\alph*}]

\item For every computable sequence $r_n$ of rational numbers the difference $|\rho_n-r_n|$ is separated from $0$ \textup(i.e., is greater than $\varepsilon$ for some $\varepsilon>0$ and almost all $n$\textup).

\item The sequence $\rho_n$ does not converge \textup(as $n\to\infty$\textup).

\item  All limit points of $\rho_n$ are Martin-L\"of random relative to~$\mathbf{0}'$.

\item The $\limsup$ of $\rho_n$ is upper semicomputable relative to~$\mathbf{0}'$ \textup(and Martin-L\"of random relative to $\mathbf{0}'$, see the previous claim\textup). 

\item The converse holds\textup: every real in $[0,1]$ that is upper semicomputable relatively to~$\mathbf{0}'$ and Martin-L\"of random relative to $\mathbf{0}'$ is the $\limsup$ of  $\rho^V_n$ for some optimal machine~$V$.  The machine $V$ can be made effectively optimal, too.
\end{enumerate}
\end{thm}

\noindent Note that \(a -- \(d\ are true for every optimal $U$ (it is easy, for example, to construct an optimal~$U$ such that corresponding sequence $\rho_n$ does not converge, but we claim that this happens for arbitrary optimal $U$).

\begin{proof}\hfill
\(a{} We use the same complexity bound $\KS(H_n\cnd n)\ge n-O(1)$ for the number $H_n$ of strings of length at most $n$ in the domain of $U$ (see Theorem~\ref{th:CNSS} and its proof). If the difference between $r_n$ and $\rho_n$ does not exceed $2^{-d}$, then 
$$
\KS(H_n\cnd n)\le \KP(d\cnd n)+(n-d)+O(1)
$$
(to specify $H_n$ given $n$ we provide a self-delimited program for
$d$ and attach $n-d+O(1)$-bits for the difference between $r_n2^{n+1}$
and $H_n$), which gives $d=O(1)$.

(b) This is a simple corollary of (a): consider a computable sequence $r_n$ for which all the points in $[0,1]$ are limit points. If $\rho_n$ converges to some $\rho$, then $\rho_n$ is close to $\rho$ for all large $n$, and $r_n$ is close to $\rho$ infinitely often, and we get a contradiction with (a).

(c) For this proof we need to use a theorem by Miller~\cite{Miller2004} (see also \cite{BMSV2012} for a simple proof): a real number  $x \in [0,1]$ is Martin-L\"of random relative to~$\mathbf{0}'$ if and only if, when viewing $x$ as a bit sequence, there is a constant~$c$ such that for every prefix $\sigma$ of $x$, there is a finite string~$\tau$ extending~$\sigma$ such that $\KS(\tau) \geq |\tau|-c$. 

Suppose~$x$ is a limit point of~$\rho_n$. First note that~$x$ cannot be a rational number (otherwise the constant sequence $r_n=x$ approximates $\rho_n$), so $x$ has a unique binary representation. Let~$\sigma$ be a prefix of $x$ and let $k$ be the length of $\sigma$. Split $[0,1]$ into $2^k$ equal intervals of size $2^{-k}$. Then~$x$ is strictly inside one of these intervals (this interval consists of all binary extensions of~$\sigma$). Since $x$ is a limit point, some $\rho_n$ also belongs to this interval. Recall that $\rho_n$ is a binary fraction $H_n/2^{n+1}$ (here it is important that we use this denominator, not $2^{n+1}-1$; of course, this does not change the limit points). Therefore, $H_n$ (considered as a string of length $n+1$ with leading zeros) is an extension of $\sigma$, and $\KS(H_n)\ge |H_n|-O(1)$ due to Theorem~\ref{th:CNSS}, so it remains to use Miller's result. 

Item (c) is proven. We do not know whether the converse holds, i.e., whether any real that is Martin-L\"of random relative to $\mathbf{0}'$ is a limit point of some sequence $\rho^V_n$ for some optimal machine~$V$. However, in (d) we give a full characterization of the reals that are $\limsup$'s of those sequences. 

(d) Let us prepare ourselves by considering a simpler question. Assume that $X$ is an arbitrary enumerable set, i.e., the domain of some machine, not necessarily an optimal one; $x_n$ is the number of strings of length exactly $n$ in $X$, and $X_n$ is the number of strings of length at most $n$ in $X$ (so $X_n=x_0+\ldots+x_n$). Consider the upper density of $X$, i.e., $\limsup X_n/2^{n+1}$. Which reals can appear as upper densities of enumerable sets?

\begin{lem}\label{lem:crit}
A real number $x$ in $[0,1]$ is the upper density of some enumerable set $X$ if and only if $x$ is upper semicomputable relative to $\mathbf{0}'$.
\end{lem}

\begin{proof}[Proof of Lemma~\ref{lem:crit}]
In one direction: $u_n=X_n/2^{n+1}$ is a uniformly lower semicomputable sequence of reals, and one can easily show that the $\limsup$ of such a sequence is upper semicomputable relative to $\mathbf{0}'$. Indeed, $\limsup u_n<r$ means that there exist $r'<r$ and some $N$ such that $u_n \le r'$ for all $n\ge N$, so we get a $\exists\forall$-property or $r$ that corresponds to a $\mathbf{0}'$-enumerable set of rational numbers.

Reverse direction: assume that $x$ is upper semicomputable relative to $\mathbf{0}'$. It is known that in this case $x$ can be represented as $\limsup k_n$ for some computable sequence $k_n$ of rational numbers (see \cite{DowneyJockuschSchupp} or \cite{VUS2013}). Then $x=\lim_n \overline{v}_n$, where $v_n=\sup(k_n,k_{n+1},\ldots)$ form a uniformly lower semicomputable sequence. We may assume without loss of generality that $v_n\in[0,1]$ (since the limit is in $[0,1]$) and that $v_n$ are rational numbers with denominator $2^n$ (by rounding; note that the resulting sequence $v_n$ may not be computable, only lower semicomputable). Then we consider an enumerable set $X$ that contains exactly $v_n2^n$ strings of length $n$ (here we use that $v_n$ are lower semicomputable). It is easy to see that the upper density of $X$ is $x$; in fact, the density (the limit, not only $\limsup$) exists and is equal to $x$, since the fraction of $n$-bit strings in $X$ converges to $x$ as $n\to\infty$.
\end{proof}

It remains to show that for $x$ that are not only upper semicomputable relative to $\mathbf{0}'$ but also Martin-L\"of random relative to $\mathbf{0}'$, the set $X$ can be made a domain of an optimal machine. Our next step is the following simple observation.

\begin{lem}\label{lem:bit-add}
If some real $x\in[0,1]$ is the upper density of the domain of some optimal machine, the same is true for $x/2$ and $(1+x)/2$.
\end{lem}
(In terms of binary representation $x/2$ is $0x$, and $(1+x)/2$ is $1x$.) 

\begin{proof}
For $x/2$ we just ``shift'' the domain of the optimal machine by adding leading~$0$'s to all the arguments. For $(1+x)/2$ we do the same and also add all strings starting with $1$ to the domain (with arbitrary values, e.g., they all can be mapped to the empty string). In both cases the machine remains optimal, the complexity increases only by $1$.
\end{proof}

Deleting the first bit preserves randomness and semicomputability, so we may assume without loss of generality that $x$ (which is random and upper semicomputable relative to $\mathbf{0}'$) is smaller than $1/2$ (starts with $0$), and then apply Lemma~\ref{lem:bit-add} to add leading $1$'s.

Now we are ready to use another known result: \emph{every random upper semicomputable $x$ is Solovay complete among upper semicomputable reals} (all properties are considered relative to $\mathbf{0}'$); according to one of the equivalent definitions of Solovay completeness, this means that for every other upper semicomputable (relative to $\mathbf{0}'$) $y$ and for any large enough integer $N$ there exists another upper semicomputable (relative to $\mathbf{0}'$) $z$ such that  $x={y}/{N}+z$.

This result combines the work of Calude et al.~\cite{CaludeHKW1998} and Ku\v cera-Slaman~\cite{KuceraS2001} (see~\cite{BienvenuS2012} for a survey of these results). Technically these papers consider lower semicomputable reals instead of upper semicomputable ones, but randomness is stable under sign change, so this does not matter. Also we need a relativized version of their result; as usual, relativization is straightforward. 

So let us assume that $x\in(0,1/2)$ and $x=y/2^d+z$ where $y$ is the upper density for some optimal machine $U$ and $z$ is upper semicomputable relative to~$\mathbf{0}'$. (The large denominator $N$ is chosen to be a power of $2$.)  Now we combine two tricks used for Lemmas~\ref{lem:crit} and~\ref{lem:bit-add}. Namely, we apply Lemma~\ref{lem:crit} to $2z$ (note that $z<1/2$), and then add leading $1$'s to all the strings in the corresponding set. This gives us density $z$ while using only right half of the binary tree (strings that start with $1$). Then we add $d$ zeros to all strings in the domain of $U$ as we did when proving Lemma~\ref{lem:bit-add}; note that $d\ge 1$ since $x<1/2$. This gives us density $y/2^d$ using only the left half of the binary tree (actually, a small part of it, if $d$ is large). Then we combine both parts and get a machine $V$ that is optimal (since the left part is optimal) and has upper density $y/2^d+z$ as required. (Note that in general $\limsup$ is not additive, but in our case we have not only $\limsup$, but limit in one of the parts, so additivity holds.)

If we start with an effectively optimal machine $U$, the machine $V$ will also be effectively optimal (for obvious reasons).

Theorem~\ref{th:classification} is proven.
\end{proof}

\section{Approximations as mass problems}\label{sec:mass}

The approach to probabilistic computations used in Theorem~\ref{th:main-prob} looks natural from a computer science perspective. However, computability theorists would probably prefer another approach inspired by the notion of mass problems. A \emph{mass problem} is a set of total functions $\mathbb{N}\to\mathbb{N}$ (in other words, a subset of the Baire space $\mathbb{N}^\mathbb{N}$ of integer sequences). Its elements are called \emph{solutions} of this problem. A problem is \emph{solvable} if it has at least one computable solution; a problem $A$ is \emph{reducible} to $B$ if there is an oracle machine $\Gamma^X$ such that $\Gamma^X\in A$ whenever $X\in B$. The full name of this reducibility is \emph{Medvedev reducibility} or \emph{strong reducibility} (compared to \emph{weak reducibility}, or \emph{Muchnik reducibility}, where the oracle machine $\Gamma$ may depend on $X\in B$).  Many people studied the corresponding degree structure, called \emph{Medvedev lattice}.

It is interesting to study mass problems from the viewpoint of probabilistic computations. For example, one may ask whether a mass problem can be solved with probability $1$ by some randomized algorithm (such an algorithm, given random oracle, produces a solution with probability $1$). One may also ask whether a mass problem can be solved with positive probability. (Both properties are downward closed with respect to strong reducibility; the second one is closed also with respect to weak reducibility.)

The notion of a mass problem allows us to reformulate the results of Section~\ref{sec:approximation} in the following way. Let $A$ be a set that we want to decide (approximately). If $\alpha$ is a total function, we may measure how well $\alpha$ approximates the characteristic (indicator) function $\chi_A$ of $A$. For a given $n$ we consider $\varepsilon_n$, the fraction of strings of length at most $n$ where $\alpha$ and $\chi_A$ differ, and then we may define $\overline{E}(\alpha,A)=\limsup \varepsilon_n$ and $\underline{E}(\alpha,A)=\liminf \varepsilon_n$. Now we may consider, for a given set $A$, the mass problem 
$$
C(A)=\{\alpha: \overline{E}(\alpha,A)=0\};
$$
its solutions are total functions that compute $\chi_A$ on a set of density $1$. (As usual, we identify binary strings with integers, so $\alpha$ is both a function defined on strings and a sequence as required by the definition of mass problems. The density of a set $X$ of strings is defined as the limit of the fraction $\varkappa_n$ of strings of length at most $n$ that belong to $X$. If this limit does not exist, we may still speak about \emph{lower density} and \emph{upper density} defined as $\liminf \varkappa_n$ and $\limsup\varkappa_n$. )

If $C(A)$ has computable solutions, the set $A$ is called \emph{coarsely computable}. Theorem~\ref{th:main} implies that the domain of an optimal machine is \emph{not} coarsely computable. Moreover, the statement (a) of this theorem guarantees that an (easier) mass problem 
$
\{\alpha: \underline{E}(\alpha,A)=0\}
$
has no computable solution either.

Here the difference with the probabilistic setting becomes apparent: while the first problem $C(A)$ cannot be solved by a probabilistic algorithm with positive probability (see below Theorem~\ref{th:prob-neg}), the second one can, as the following result shows:

\begin{thm}\label{th:prob-pos}
For every enumerable set $A$ there exists a randomized algorithm that with positive probability computes a total function $\alpha$ such that $\underline{E}(\alpha,A)=0$.
\end{thm}

In other terms, there exists an oracle machine $\Gamma$ such that the set of oracles $X$ for which $\Gamma^X$ computes a total function $\alpha$ such that $\underline{E}(\alpha,A)=0$ has positive probability according to the uniform Bernoulli distribution on the Cantor space (corresponding to independent fair coin tosses).

\begin{proof}
To make the proof shorter, we use known results about generic sequences. A sequence~$G$ in the Baire space is called \emph{generic} (more precisely, $1$-generic) if for every effectively open set $U$ it is contained either in $U$ or in the interior part of the complement of $U$.  In other terms, for every enumerable set $W$ of finite sequences of integers, either $G$ has a prefix in $W$, or $G$ has a prefix that has no extensions in $W$. It is easy to see (using the Baire category theorem) that generic sequences exists. Moreover, as Kurtz has shown (see~\cite[Theorem 8.21.3]{DowneyH2010}), generic sequences can be generated by a randomized algorithm with positive probability. So it remains to prove the following lemma:

\begin{lem}\label{lem:generic-use}
There exists an oracle machine $\Psi$ such that for every generic $G$ the machine $\Psi$ with oracle $G$ computes a total function $\alpha$ with $\underline{E}(\alpha,A)=0$.
\end{lem}

\begin{proof}[Proof of the Lemma~\ref{lem:generic-use}]
The function $\alpha$ is constructed as follows: we split all the possible lengths into consecutive intervals by thresholds $n_0<n_1<n_2<\ldots$, and for lengths in $[n_{i-1},n_{i})$ we run the enumeration of $A$ for $N_i$ steps (and use the resulting part of $A$ for $\alpha$). Here $n_i$ and~$N_i$ are parameters that are taken from the generic sequence $G$; we interpret the $n$th term $G_n$ as an encoding of some pair $(n_i,N_i)$.

Why does this help? Let us prove that for infinitely many values of $n$ the fraction of errors among strings of length at most $n$ is small. Let $\rho_n$ be the fraction of strings of length at most~$n$ that belong to $A$, and let $\rho$ be the $\limsup$ of $\rho_n$. Consider a rational number $r$ that is strictly smaller than $\rho$ but very close to it. Whatever the number $n_0,\ldots, n_{i-1}$ and $N_0,\ldots,N_{i-1}$ (used in the construction of $\alpha$) are, it is always possible to choose $n_i$ and $N_i$ in such a way that the density of $\alpha$ among strings of length at most $n_i$ exceeds $r$. This inequality, considered as a property of finite sequence $(n_0,N_0),\ldots,(n_i,N_i)$, is (computably) enumerable, so we get a dense effectively open set. According to the definition of genericity, the sequence $G$ should have prefix in this effectively open set (its complement has an empty interior). The same argument can be applied to sequences of lengths greater than some threshold, so we conclude that for generic $G$ the corresponding sequence $\alpha$ has infinitely many prefixes for which the error density $\varepsilon_n$ is small. 

There is a small technical detail in this argument: we need to ensure that $n_i\gg n_{i-1}$, so the strings of length at most $n_{i-1}$ form an asymptotically negligible fraction in the set of all strings of length at most~$n_i$. But this is easy to achieve: we may require, for example, that $n_i>n_{i-1}+i$.
\end{proof}

Theorem~\ref{th:prob-pos} is proven.
\end{proof}

The statement (b) of Theorem~\ref{th:main} remains true for randomized algorithms. Consider again an optimal machine $U$ and its domain $\dom(U)$.

\begin{thm}\label{th:prob-neg}
There exists some $\varepsilon>0$ such that no randomized algorithm computes $\alpha$ with $\overline{E}(\alpha,\dom(U))<\varepsilon$ with positive probability.
\end{thm}

The value of $\varepsilon$ with this property depends on the choice of the optimal machine~$U$. In this statement we may allow $\alpha$ to be non-total (and consider the places where $\alpha$ is undefined, as error points). This makes the interpretation of this result in terms of mass problems a bit more difficult, since mass problems are (by definition) sets of \emph{total} functions. But it is nevertheless possible: we can consider the mass problem of enumerating the graph of a partial function $\alpha$ that approximates $\dom(U)$. For the proof however, this interpretation is not needed, and we do not go into details.

\begin{proof}
For this proof we need to extend the inequality used to prove Theorem~\ref{th:main}.

\begin{lem}\label{lem:most-general}
Consider some oracle machine $\Gamma$  with random oracle $X$ that tries to decide $\dom(U)$.  For some length $n$ and for some threshold $d$ consider the event: ``the fraction of errors made by $\Gamma^X$ on strings of length at most $n$ is at most $2^{-d}$''\textup; let $p(\Gamma,n,d)$ be its probability. Then
$$
\KP(H_n\cnd n)\le O(\log d)+(n-d)+\KP(\Gamma\cnd n)+\log\frac{1}{p(\Gamma,n,d)} 
$$
where $H_n$ is the number of strings of length at most $n$ in $\dom(U)$.
\end{lem}

For deterministic algorithms (and $p=1$) we have seen this bound earlier in the proof of Theorem~\ref{th:main} with $\KS(H_n\cnd n)$ instead for $\KP(H_n\cnd n)$ (we did not need $\KP(H_n\cnd n)$ there, but the same argument would work for $\KP$ too). Now the prefix-free version of complexity is technically convenient since it is better adapted to oracles and probability.

\begin{proof}[Proof of Lemma~\ref{lem:most-general}]
Let $X$ be a ``good'' value of the oracle for which this event happens. Assume that we know $n$ and $d$ is given as advice. Then we wait until the machine $\Gamma^X$ provides answers for all strings of length at most $n$ except for a $2^{-d}$ fraction, and count the number of positive answers. This number differs from $H_n$ by $O(2^{n-d})$, and the difference can be specified by $n-d+O(1)$ bits (in a self-delimited way when $n$ and $d$ are known, so we can put prefix complexity $\KP$ in the left-hand side). So we get
$$
\KP^X(H_n\cnd n)\le O(\log d)+(n-d)+\KP(\Gamma\cnd n) 
$$
(complexity is relativized by $X$, the length $n$ is given as a condition, and a self-delimiting description of $d$ uses only $O(\log d)$ bits). It remains to apply a general statement about complexity: \emph{if for fixed $a$, $b$, $k$ and for random $X$ the inequality $\KP^X(a\cnd b)\le k$ holds with probability at least~$p$, then $\KP(a\cnd b)\le k+\log\frac{1}{p}+O(1)$}. This statement can be easily proven by switching to a priori probabilities and  taking the average over random oracles $X$.
\end{proof}

To prove the theorem, we need to find $\varepsilon>0$ such that no randomized algorithm can compute~$\alpha $ such that $\overline{E}(\alpha,\dom(U))<\varepsilon$ with positive probability. First of all we note that the Lebesgue density theorem implies that every set of positive measure forms a majority in some interval, and we can consider only oracles from this interval. So we can replace ``with positive probability'' by ``with probability greater than $1/2$'' without loss of generality. (Any other constant instead of $1/2$ will work, too, but it is important to have some fixed constant since the choice of $\varepsilon$ should not depend on the success probability.)

Assume that for some $d$ and for some oracle machine $\Gamma^X$ with random oracle $X$ the event ``\emph{$\limsup$ of error rate of $\Gamma^X$ is less than $2^{-d}$}'' has probability at least $1/2$. For each $n$ there is a set $T_n$ of oracles that give an error rate of at most $2^{-d}$ on strings of length at most $n$. The event mentioned above (that has probability at least $1/2$) is included in $\liminf T_n =\cup_{N}\cap_{n\ge N} T_n$. If this event has probability at least $1/2$, then some intersection $\cap_{n\ge N} T_n$ has probability at least $1/3$. This implies that probabilities of $T_n$ for all sufficiently large $n$ are at least $1/3$. Applying Lemma~\ref{lem:most-general} and the inequality $\KP(H_n\cnd n)\ge \KS(H_n\cnd n)\ge n$ (that holds with $O(1)$-precision), we conclude that $d-O(\log d)\le \KP(\Gamma\cnd n)+O(1)$ for sufficiently large $n$, where the constant in $O(1)$ does not depend on $\Gamma$ and $n$. Now, knowing this constant and recalling that for every $\Gamma$ the complexity $\KP(\Gamma\cnd n)$ is $O(1)$ for infinitely many $n$, we may choose $\varepsilon$ small enough, and finish the proof of Theorem~\ref{th:prob-neg}.
\end{proof}

For the last result in this section we should recall two paradigms of approximate computation that have received a lot of attention in the recent literature, the so-called \emph{coarse computability} and \emph{generic computability} (see~\cite{DowneyJockuschSchupp}). We already considered, given some set $A$, the mass problem ``coarsely compute $A$''  that consists of total functions $\alpha$ that coincide with $\chi_A$ on a set of density~$1$. If this problem has a computable solution, i.e., if there exists a total computable function $\alpha$ that coincides with $\chi_A$ on a set of density $1$, the set $A$ is called \emph{coarsely computable}.
{\looseness=1

}
 Generic computability of $A$ is defined in a similar way, but now wrong answers are not allowed, and $\alpha$ may be a non-total function. As we have discussed, this does not give a mass problem in Medvedev sense directly, but we may consider a mass problem of enumerating the graph of~$\alpha$. If this problem is solvable, there exists a computable partial function $\alpha$ defined on a set of density $1$ that never gives wrong answers. In this case $A$ is called \emph{generically computable}.

These two notions of approximate computability are incomparable: an enumerable set can be coarsely computable but not generically computable and vice-versa (see~\cite{DowneyJockuschSchupp}).

Using this language, we can say that Theorem~\ref{th:prob-pos} is about coarse computability (we get there a total approximating function with errors) while Theorem~\ref{th:prob-neg} can be applied to both notions (and even to their combination where we allow partial functions and errors at the same time). Now we state two results about the generic model:

\begin{thm}\label{th:prob-generic}\hfill
\begin{enumerate}[label=\({\alph*}]

\item Let $U$ be an optimal machine. There is no randomized algorithm which with positive probability computes a partial function $\alpha$ that makes no errors for $\dom(U)$ and is defined on a set of upper density $1$.

\item Whether $\dom(U)$ is $(1-\varepsilon)$-i.o.-generically probabilistically computable or not depends on the particular choice of machine~$U$.
\end{enumerate}
\end{thm}

\noindent The upper density requirement in (a) means that $\underline{E}(\alpha,\dom(U))=0$, so this result shows that the permission to give (rare) wrong answers was crucial for Theorem~\ref{th:prob-pos}. Note that the randomized algorithm in this result is still permitted to make errors (with positive probability), but we require that with positive probability it computes a function without errors (and with dense domain).

\begin{proof}
(a)~Using again the density argument, we may assume without loss of generality that a randomized algorithm $\Gamma$ computes some partial function with the required properties with probability close to $1$, say, greater than $0.9$. Let $\mathcal{G}$ be the set of good oracles; the measure of $\mathcal{G}$ exceeds $0.9$ and for every $X\in \mathcal{G}$ the machine $\Gamma$ with oracle $X$ computes some partial function that has dense domain and makes no errors for $\dom(U)$. Note that for different oracles $X\in\mathcal{G}$ we may get different partial functions with these properties. 

We simulate the behavior of $\Gamma$ on all possible random bits looking for the oracle prefixes that guarantee the algorithm's answers for most strings of length at most $n$, for some sufficiently large $n$. More precisely, for some integer $d$ we are looking for triples $(n,a,x)$ where $n$ is an integer (length), $a$ is a partial $0$-$1$-valued function on strings of length at most $n$, and $x$ is a bit string (interpreted as an oracle's prefix),  such that
\begin{itemize}
\item $n\ge d$;
\item $a$ is defined on at least $(1-2^{-d})$-fraction of strings of length at most $n$;
\item $x$ guarantees answers given by $a$: for every oracle $X$ that extends $x$, the machine $\Gamma$ with oracle $X$ computes some extension of $a$.
\end{itemize}

\noindent The first requirement ($n\ge d$) guarantees that the $2^{-d}$-fraction mentioned in the third requirement makes sense for strings of length at most $n$.
For every $d$ and for every oracle $X$ in $\mathcal{G}$ there exists some prefix $x$ and some $n$ and $a$ with these properties (since $\Gamma^X$ has domain of upper density $1$). By compactness, for every $d$ one can effectively find a finite set of triples $(n_1,a_1,x_1),\ldots,(n_k,a_k,x_k)$ such that 
\begin{itemize}
\item $n_i\ge d$ for all $i$; 
\item the intervals $\Omega_{x_i}$ (containing extensions of $x_i$) cover more than $90\%$ of the Cantor space;
\item oracle prefix $x_i$ guarantees output $a_i$;
\item $a_i$ is defined for a $(1-2^{-d})$-fraction of all strings of length at most~$n_i$. 
\end{itemize}
One may assume that all intervals $\Omega_{x_i}$ are disjoint: we may split large intervals into small parts and delete the repetitions.

Note that there is no guarantee that the answers provided by $a_i$ are correct for $\dom(U)$: we know that $\Gamma$ produces correct approximations for oracles in $\mathcal{G}$, but some of the intervals $\Omega_{x_i}$ could be outside $\mathcal{G}$. Note, however, that \emph{if} $\Omega_{x_i}$ \emph{has non-empty intersection with $\mathcal{G}$, then the answers in $a_i$ are correct for $\dom(U)$}.

Now we classify all the triples $(n_i,a_i,x_i)$ according to $n_i$: the triples that have the same $n_i$ are put into one group. Consider some group that corresponds to some length $n$. For different triples in this group the answers $a_i$ may be inconsistent. Still, we try to guess the number of strings of length at most $n$ in $\dom(U)$, using some kind of majority voting. We count the number of positive answers in $a_i$ for each triple in the group. If there exists some $u$ such that most triples in the group have the number of positive answers in $2^{n-d}$-neighborhood of $u$, we choose some $u$ with this property and declare it to be our guess for length $n$. In the last sentence the term ``most triples'' is understood in the measure sense: the corresponding intervals $\Omega_{x_i}$ should cover more than $50\%$ of the total measure of all intervals in the group.

\begin{lem}
For at least one group \textup(corresponding to some length $n$\textup) the guess is made and is $O(2^{n-d})$-close to the number of strings of length at most $n$ in $\dom(U)$.
\end{lem}

Note that for other groups no guess (or a wrong guess) could be made.

\begin{proof}
The set $\mathcal{G}$ has measure greater than $0.9$, so its complement has measure less than $0.1$. On the other hand, the union of all $\Omega_{x_i}$ has measure at least $0.9$, so at least $8/9$ of it consists of oracles in $\mathcal{G}$. So for some length $n$ at least $8/9$ of the intervals of the corresponding group have non-empty intersection with $\mathcal{G}$. For these intervals the correct answer belongs to the $2^{n-d}$ neighborhoods of the numbers computed from $a_i$, so the guess is made, and it is $O(2^{n-d})$-close to the correct answer. 
\end{proof}

Now note that we have described the process that is effective when $d$ is given. Therefore, for each length $n$ the guess (if it is made for this length) has conditional complexity (given $n$) at most $O(\log d)$. If it is $O(2^{n-d})$ close to the correct number ($H_n$), then
$$
\KS(H_n\cnd n)\le n-d+O(\log d)
$$
On the other hand, $\KS(H_n\cnd n)\ge n-O(1)$, so for large $n$ we get a contradiction.

(b)~The positive part of (b) is true even without randomization, as we have seen in Theorem~\ref{th:main-prob}~(b). The negative part can be proven in the same way as Theorem~\ref{th:main}~(d), but we need a more sophisticated tool instead of a sparse simple set. Here it is.

\begin{lem}\label{lem:supersimple}
There exists an enumerable set $S$ of density $0$ with the following property: for every randomized algorithm $\Gamma$ the probability of the event ``$\Gamma$ enumerates an infinite subset of the complement of $S$'' is zero.
\end{lem}

The first impression is that the statement of Lemma~\ref{lem:supersimple} is false. Indeed, if the density of $S$ is small, then choosing uniformly a random element in $1\ldots N$ for large $N$, we get an element outside $S$ with probability close to $1$. Then we repeat this procedure with much larger $N$ and much smaller probability of error, etc. If the sum or error probabilities is less than $1$, in this way we (with positive probability) enumerate an infinite set outside $S$.

What is wrong with this argument? We implicitly assumed here that the density computably converges to $0$ (when choosing $N$ large enough to make error probability small). So the set $S$ with required properties may exist (and it does!), though the convergence for density cannot be computable.

\begin{proof}[Proof of Lemma~\ref{lem:supersimple}]
Due to the Lebesgue density argument, it is enough to show (for some enumerable set $S$) that \emph{no randomized algorithm can enumerate an infinite set outside $S$ with probability more than $1/2$}.  This statement can be split into a sequence of requirements $R_e$, where $R_e$ says that $e$th randomized algorithm $\Gamma_e$ fails to do that. And, of course, we have density $0$ requirement (in addition to all $R_e$).

We take care of $R_e$ separately for each $e$, having some process that adds some strings to~$S$. Note that the processes for different $e$ may only help each other (the bigger $S$ is, the more difficult it is to enumerate an infinite set outside $S$), except for the density requirement: we need to ensure that all the strings added to $S$ by all processes (for all $e$) still form a set of density~$0$. To take care of densities, let us take a convergent series of rational numbers with $E=\sum_e \varepsilon_e<1$, and agree in advance that $e$th process may add to $S$ at most $\varepsilon_e$-fraction of strings of length $k$ for every length $k$. For example, we may let $\varepsilon_e=2^{-(e+2)}$. This guarantees that the density of $S$ does not exceed $E$, and some additional argument will show that it is actually $0$.

Now we describe the $e$th process (that takes care of $R_e$). There is a technical problem: if $\varepsilon_e$ is small, the permission to use an $\varepsilon_e$-fraction of strings for each length does not allow use to use short strings, because even one string added may exceed the threshold. Since we are interested in the enumeration of an \emph{infinite} set outside $S$, we may ignore short strings and start with some length $N_e$ where we are allowed to add at least an $\varepsilon_e/2$-fraction of the strings of that length not crossing $\varepsilon_e$-threshold. So without loss of generality we may assume that $\Gamma_e$ generates only strings of length at least $N_e$.
 
We simulate $\Gamma_e$ on all oracles to find which strings appear as \emph{first} elements of the enumeration and what are their probabilities. In this way we get a lower semicomputable probability distribution on strings of length at least $N_e$. If the total probability (the sum of probabilities for all possible output strings) never exceeds $1/2$, we do not do anything and $R_e$ is vacuously satisfied.

As soon as the sum of probabilities exceeds $1/2$, we stop the simulation and choose which strings should be added to $S$. For each $n\ge N_e$ we sort all the strings of length $n$ in order of decreasing probability and add to $S$ the first $\lceil \varepsilon_e/2\rceil 2^n$ of them (so the fraction is at least $\varepsilon_e/2$ and at most $\varepsilon_e$, due to the assumption about $N_e$). 

What do we achieve in this way? We cover by $S$ some set of strings that have total probability (in the $\Gamma_e$-enumeration) at least $\varepsilon_e/4$. Indeed, for each length we cover at least a $\varepsilon_e/4$-fraction of total probability for this length (by taking $\varepsilon_e/4$-fraction of most probable strings), and the sum of these total probabilities for all lengths is at least $1/2$. Then we increase $N_e$, making it greater than all strings we have added to $S$, carve out from the probability space the part corresponding to covered elements, and repeat the process. Since our goal is to prevent enumerating an infinite set outside $S$, the increase in $N_e$ does not matter. The second stage (if the probability reaches $1/2$ again) carves out additional $\varepsilon_e/4$ from the probability of enumerating an infinite set outside $S$. (Note that we do not simulate $\Gamma_e$ further for the parts of the probability space already carved out.) And so on --- the number of iterations is bounded, since each time we decrease the probability by at least $\varepsilon_e/4$, and at some stage the probability never exceeds $1/2$, and $R_e$ is satisfied.

It remains to explain why the resulting $S$ (the union of all elements added by all the processes) has zero density. Since we add only finitely many elements to $S$ for each of the requirements $R_e$, we know that starting from some point, the density of $S$ is bounded by the tail of the series $\sum_e\varepsilon_e$. This tail can be arbitrarily small, so the density of $S$ is zero. 
\end{proof}

The set provided by Lemma~\ref{lem:supersimple} allows us to finish the proof of Theorem~\ref{th:prob-generic}: it is embedded with positive density in the domain of every effectively optimal machine, and for this part the probabilistic algorithm either makes errors or generates only a sparse set, so the density of the domain is separated from $1$ if no errors are made.
\end{proof}

\section{Busy beavers and fraction of long computations}\label{sec:busy-beaver}

We define the ``busy beaver'' function $\BB(n)$ as the maximal running time for terminating computations of the optimal machine on strings of length at most $n$. This functions increases very fast (faster than any computable function) and provides some invariant scale for measuring the running time of computations in that it does not depend significantly on the choice of the optimal machine and computational model used to measure the running time. Namely, the following invariant definition is possible:

\begin{prop}\label{prop:b-vs-bb}
Let $B(n)$ be the maximal integer of complexity at most $n$\textup; then $$B(n-c)\le \BB(n) \le B(n+c)$$ for some $c$ and for all $n$.
\end{prop}

(See~\cite{VUS2013} for the proof and more details.) 
 
 So we know that if we run the optimal machine on all the inputs of size at most $n$, we need to make $\BB(n)$ steps (for each input) until all the terminating computations terminate. A natural question arises: how long should we wait until \emph{most of the computations} (say, except for $2^{n-k}$, i.e., $2^{-k}$-fraction) terminate?  The following theorem gives an answer with logarithmic precision (note that both $\KP(k\cnd n)$ and $\KP(n)$ are logarithmic in $n$):
 
 \begin{thm}\label{th:busy-beavers}\hfill
\begin{enumerate}[label=\({\alph*}]

\item If after $t$ steps at most $2^{n-k}$ terminating computations are still running, then $$t\ge \BB(k-\KP(k\cnd n)-O(1)).$$

\item If after $t$ steps more than $2^{n-k}$ terminating computations are still running, then $$t\le \BB(k+\KP(n)+O(1)).$$
\end{enumerate}
 \end{thm}

\begin{proof}
(a)~Since $\BB(i)$ (up to $O(1)$-change in the argument, see Proposition~\ref{prop:b-vs-bb}) can be defined as the maximal number of complexity at most $i$, we need to show only that for every such $t$ we have $\KS(t)\ge k-\KP(k\cnd n)-O(1)$.

Indeed, to reconstruct $H_n$ given $n$, it is enough to know $t$ and the number~$N$ of terminating computations (on inputs of length at most $n$) that are still running after $t$ steps. To encode this information, we start by a self-delimited description for $k$ given $n$ (using $\KP(k\cnd n)$ bits), then append~$N$ written in binary (using an $(n-k)$-bit string), and finally append a $\KS(t)$-bit description of $t$. Knowing $n$ and this encoding, we first find $k$, then read the next $n-k$ bits (that form some number $U$), use the rest to reconstruct $t$, and then make $t$ steps for each input of length at most $n$, count the terminated computations and add $U$ to get $H_n$. Therefore, $\KP(k\cnd n)+n-k+\KS(t)\ge \KS(H_n\cnd n)\ge n-O(1)$, and we get the desired inequality.

(b)~To prove this bound, let us construct a number of complexity at most $k+\KP(n)$ that exceeds $t$. Consider the number $H_n$ of terminating programs of length at most~$n$, and consider its first $k$ bits (i.e., $H_n$ where the last $n-k$ bits are replaced by zeros). Knowing this string and $n$, we can wait until that many terminating computations appear, and get the number $t$ of steps needed for that. In total it is enough to know $\KP(n)+k$ bits, as we promised.
\end{proof}

One would like to get rid of the logarithmic terms in these statements. Sometimes it is indeed possible. For example, in (b) we can replace $\KP(n)$ by $O(\KS(n\cnd k))$. Indeed, a prefix-free description of a program that maps $k$ to $n$ uses $O(\KS(n\cnd k))$ bits, and if we append the first~$k$ bits of $H_n$ to it, we can then reconstruct the prefix-free description, then~$k$, then $n$, and finally $t$.\footnote{Technical remark: it is important that the description is self-delimiting when $k$ is not known, so this argument does not allow us to write $\KP(n\cnd k)$ instead of $O(\KS(n\cnd  k))$.} This implies, for example, the following corollary:

\begin{cor}
Making $\BB(n/2)$ steps of computation of the universal machine for each input of length at most $n$, we have $2^{n/2\pm O(1)}$ unfinished \textup(terminating\textup) computations.
\end{cor}

Still it would be nice to find a more general result with $O(1)$-precision (or to show that logarithmic terms are unavoidable in the general case), which we leave as an open question. 

\section{Some generalizations}\label{sec:appendix}

\subsection{Beyond the complexity arguments: coarse separation}

Most of the arguments above are based on the lower bound for the Kolmogorov complexity of $\KS(H_n\cnd n)$. Still there are natural questions of similar type where the complexity argument does not help us much (at least not in an obvious fashion). 

Recall the concept of enumerable inseparable sets, which we already used in Section~\ref{sec:opt-and-eff-opt}: two enumerable sets $A$ and $B$ are inseparable if there is no \emph{separator}, i.e., no \emph{total} computable function $h$ (defined on strings and taking $0/1$-values) such that $h(A)=0$ and $h(B)=1$. In other words, every computable function makes errors --- there are some error points $x$ such that either $h(x)$ is undefined, or $x\in A$ and $h(x)=1$, or $x\in B$ and $h(x)=0$. 

Here we are interested in the density of the error points. Is it possible that for some enumerable inseparable sets $A$ and $B$  there exists a computable (not necessarily total) function $h$ for which the error points have density $0$? As before, this does happen for certain $A$ and $B$. Assume, for example, that $A$ and $B$ consist only of strings of the form $0^m$: we take two inseparable sets of integers $U$ and $V$ and let $A=\{0^m: m\in U\}$ and $B=\{0^m: m\in V\}$. Obviously, for every total function only strings of type $0^m$ could be error points, and these strings form a set of density $0$. To make the question non-trivial, we need to consider some standard inseparable sets, as we did before for the halting problem.

\begin{thm}
Let $U$ be an effectively optimal machine. Then the sets $A=\{x\colon U(x)=0\}$ and $B=\{x\colon U(x)=1\}$ are disjoint and inseparable. Moreover, for every computable function $h$ the fraction of errors made by $h$ \textup(as a separator\textup) on strings of length at most $n$, i.e., 
    $$E_n=\frac{\#\{x\colon |x|\le n,  (h(x) \text{ is undefined }) \lor (x\in A \land h(x)=1) \text{ or } (x\in B\land h(x)=0) \}}{2^{n+1}}$$
is separated from $0$ for large $n$, i.e., $\liminf E_n >0$.
\end{thm}
(Again we should write $2^{n+1}-1$, to be exact, but this does not matter for the limit.)

\begin{proof}
For every computable function $V$ the injective length-bounded reduction $t_V$ of $V$ to an optimal machine can be found effectively. Indeed, this happens for the ``standard'' effectively optimal machine, and then we can use Schnorr's isomorphism result to conclude that this is true for any effective optimal machine. So we can construct $V$ assuming (in a fixed-point way) that the reduction $t_V$ is known in advance, and let $V(x)=1-h(t_V(x))$. Then every point in the image of $t_V$ is an error point for $h$, and they have positive lower density because $h$ is injective and length-bounded.
\end{proof}

\begin{rem}
In this result the \emph{effective} optimality is essential. Indeed, one can construct an optimal machine $U$ such that $U^{-1}(0)$ and $U^{-1}(1)$ each contain only one element. (We keep some preimage of $0$ and suppress all others; we also do the same thing with $1$.) For this optimal machine the statement is evidently false.
\end{rem}

\subsection{Schnorr-type isomorphism results}

Schnorr~\cite{Schnorr1974} considered not only effectively optimal machines (he called them ``optimal enumerations'') but also a similar notion for numberings of computable functions. Each computable function $U(\cdot,\cdot)$ of two arguments can be considered as a \emph{numbering} of some class of computable functions of one argument. Namely, for each $e$ we consider the function $U_e\colon x\mapsto U(e,x)$ and say that $U_e$ \emph{has number $e$ in the $U$-numbering}. (Since we speak about numbers, we assume that both arguments are natural numbers, not strings.)

If every computable function of one argument appears among the $U_e$'s, the function $U$ is \emph{universal} (for the class of computable functions of one argument). A universal function $U$ is called \emph{G\"odel universal} function if for every other computable function $V(\cdot,\cdot)$ the $V$-numbering can be reduced to $U$-numbering via some total computable function $h$; this means that 
$$
   V_e=U_{h(e)} \text{ for all $e$, i.e., } V(e,x)=U(h(e),x) \text{ for all $e$ and $x$}. 
$$
The last equation means that $V(e,x)$ and $U(h(e),x)$ are both undefined or both defined and equal.

Requiring the existence of a length-bounded reduction function $h$, we get the definition of a \emph{effectively optimal numbering}. Since we consider $e$ as an integer, length-bounded functions are defined as functions $h\colon \mathbb{N}\to\mathbb{N}$ such that $h(n)=O(n)$. (This is the same notion as before if we identify natural number with strings in a standard way.) One can prove (in a standard way) that effectively optimal numberings exist; corresponding functions of two arguments are also called \emph{effectively optimal}.

In~\cite{Schnorr1974}, Schnorr used the name ``optimal numbering'' for this notion, but it is natural to reserve this name for a weaker notion defined as follows. Every computable function $U(\cdot,\cdot)$ can be used to measure ``complexity'' of computable functions: if $f(\cdot)$ is some computable function, then its complexity $\KS_U(f)$ is defined as the logarithm of the minimal $U$-number of~$f$. (Logarithms of natural numbers correspond to lengths for strings.) Complexity $\KS_U(f)$ is infinite if~$f$ does not appear among $U_e$. The computable function $U(\cdot,\cdot)$ and the corresponding numbering are called \emph{optimal} if $\KS_U$ is minimal up to $O(1)$ additive term, i.e., for every computable $V(\cdot,\cdot)$ there exists some $c$ such that $\KS_U(f)\le \KS_V(f)+c$ for all $f$.  This implies that $\KS_U(f)$ is finite for every computable $f$, i.e., that $U$ is universal.

Now we return to \emph{effectively} optimal universal functions (numberings). Schnorr proved in~\cite{Schnorr1974} that for every two \emph{effectively} optimal universal functions $U(\cdot,\cdot)$ and $U'(\cdot,\cdot)$ there exists a computable bijection $h$ that is length-bounded in both directions and $U_e=U'_{h(e)}$ for all $e$, i.e., $U(e,x)=U'(h(e),x)$ for all $e$ and $x$.

Looking at three similar results (Myhill--Schnorr theorem, Schnorr's result about optimal machines, and just mentioned Schnorr's result about numberings), it is natural to ask whether they can be generalized. Indeed it is possible, in a rather straightforward way. Before explaining this general framework, let us consider one more example: numberings of $\Sigma_n$-sets.

Let us fix some $n$ and consider $\Sigma_n$-sets, i.e., sets that can be obtained from a decidable predicate by a $n$-quantifier prefix starting with $\exists$. There exists a universal $\Sigma_n$-set of pairs (this means that every $\Sigma_n$-set can be obtained as its ``vertical section'', by fixing the first component); such a set determines a numbering of $\Sigma_n$-sets.  Then we can define \emph{optimal} universal sets (that give minimal complexity function) and \emph{effectively optimal} universal sets (that provide a length-bounded computable reduction for every other $\Sigma_n$-set of pairs), 
and prove the existence of effectively optimal universal sets. 

Optimal $\Sigma_n$-sets can be used to define a complexity notion for $\Sigma_n$-sets\footnote{Similar ideas were considered in~\cite{Calude2006}.},
  and the ``Schnorr theorem for $\Sigma_n$-sets'' then says that \emph{every two effectively optimal universal sets differ by a computable bijection that is length-bounded in both directions}.

One can provide some general framework for all these examples. Let $T$ be some abstract set. We assume that one special element $\bot\in T$ is fixed, as well as some class of total mappings $\mathbb{N}\to T$, called ``enumerations''\footnote{We use the name suggested by Schnorr, though in a more general setting.}. We assume that the class of enumerations is closed with respect to partial computable reductions:
\begin{quote}
(1)~\emph{if $\nu\colon \mathbb{N}\to T$ is an enumeration, and $f\colon \mathbb{N}\to\mathbb{N}$ is a partial computable function, then the function 
$$\mu(n)=\textup{\textbf{if }} \text{f(n) is defined} \textup{\textbf{ then }} \nu(f(n)) \textup{\textbf{ else }} \bot \textup{\textbf{ fi}}$$
is an enumeration},
\end{quote}
and has a maximal element:
\begin{quote}
(2)~\emph{there exists an enumeration $\nu$ such that for every enumeration $\mu$ there is a total computable $h$ such that $\mu(n)=\nu(h(n))$ for all $n$}.
\end{quote}
Examples: 
\begin{itemize}
\item
$T=\{\bot,\top\}$; enumerations correspond to enumerable sets (elements of the set are mapped to $\top$, others to $\bot$;\footnote{Here the term ``enumeration'' sounds confusing, since the enumerable set is the domain of an enumeration, not its range.}
\item
$T=\mathbb{N} \cup \{\bot\}$; enumerations are partial computable functions where undefined values are replaced by $\bot$;
\item
$T$ is the family of computable partial unary functions; enumerations correspond to computable partial binary functions;
\item
 $T$ is the family of $\Sigma_n$-sets; enumerations correspond to $\Sigma_n$-sets of pairs.
\end{itemize}
For a given enumeration $\nu$, a \emph{complexity} function is defined: the complexity of $t\in T$ is the logarithm of the minimal $\nu$-number of $t$. As usual, identifying integers with strings, we may defined complexity as the minimal length of $p$ such that $\nu(p)=t$. For the first example this notion is meaningless, but for the other three it is reasonable. (We get Kolmogorov complexity for the second example and complexity of computable functions as defined by Schnorr for the third one.)

One can prove (by a standard argument) that there exist optimal enumerations that make the complexity minimal up to $O(1)$; one can also strengthen this statement and derive from conditions~(1) and~(2) that there exists an \emph{effectively} optimal enumeration (every enumeration is reducible to it by a total length-bounded function).

Indeed, let $\nu$ be a maximal enumeration from (2); the first property gives us an enumeration $\omega(\hat e x)=\nu([e](x))$ where $\hat e$ is the standard self-delimiting encoding of $e$, and $[e](x)$ is the output of program $e$ on input $x$; if $[e](x)$ is undefined, then $\omega(x)=\bot$. (Here we identify integers with strings and use self-delimiting encoding and concatenations.) The enumeration $\omega$ is effectively optimal: if $\mu$ is some other enumeration, then $\mu(x)=\nu(f(x))$ for some total computable $f$, since $\nu$ is maximal; if~$e$ is the program for $f$, then $\mu(x)=\nu([e](x))=\omega(\hat e x)$, so $x\mapsto \hat e x$ is a length-bounded reduction.

To apply Schnorr's argument in this general framework, one more property is needed: 
\begin{quote}
(3)~\emph{if $\tau$ is an enumeration, and $G$ is an enumerable set of integers, there exists an enumeration $\rho$ such that:
\begin{itemize}
\item if $x\notin G$, then  $\rho(x)=\tau(x)$;
\item if $x\in G$, then $\rho(x)\ne\bot$.
\end{itemize}}
\end{quote}
It is easy to check this property for all the examples above: it says, informally speaking, that at any moment we may change an object in $T$ making it different from the ``bottom'' object $\bot$. Having this property, we may state a  generalized version of Schnorr's result as follows: 

\begin{thm}\label{th:schnorr-general}
Assume that a set $T$ and a class of enumerations are fixed that satisfy the requirements \textup{(1)--(3)}. Let $\nu$ and $\nu'$ be two effectively optimal enumerations. Then there exists a computable bijection $h$ that is length-bounded in both directions such that $\nu'(x)=\nu(h(x))$ for all $x\in \mathbb{N}$.
\end{thm}

\begin{proof}

As before, it is enough to show that every enumeration can be reduced to every effectively optimal enumeration by a total computable length-bounded injection. The second combinatorial part of the proof remains unchanged. Assume that we have an effectively optimal (therefore, universal) enumeration $\omega(\cdot)$ and some other enumeration $\mu(\cdot)$. We want to reduce $\mu$ to $\omega$ by some computable injective length-bounded total function. 

 First, we construct the uniform size-bounded reduction of all enumerations. The property~(1) guarantees that there exists an enumeration $\nu$ such that
$$
\nu(\hat e x)=\begin{cases}
\omega([e](x)), \text{ if $[e](x)$ is defined};\\
\bot \text{ otherwise}.
\end{cases}
$$
Since $\omega$ is optimal, there exists a total computable function $h(e,x)$ such that $\nu(\hat e x)=\omega(h(e,x))$ and $|h(e,x)|\le |x| + c_e$ for all $x$ and $e$. Here $c_e$ depends on $e$ (but not on~$x$); in fact $|h(e,x)|\le |x|+|\hat e|+O(1)$. Recalling the construction of $\nu$, we see that
$$
\omega(h(e,x))=\begin{cases}
\omega([e](x)), \text{ if $[e](x)$ is defined};\\
\bot \text{ otherwise}.
\end{cases}
$$
We want to find $m$ such that $[m]$ is total, 
$$ \omega([m](x))=\mu(x) $$
for all $x$, and the function $x\mapsto h(m,x)$ is injective. Then this function will be the required computable injective length-bounded reduction of $\mu$ to $\omega$.

To find such an $m$, we construct some total function $M(m,x)$ and use the fixed-point theorem to conclude that $M(m,x)=[m](x)$ for some $m$ and for all $x$. This function $M(m,x)$ will have the following properties:
\begin{itemize}
\item if $h(m,x)=h(m,x')$ for some $x'<x$, then $M(m,x)$ is some (fixed) preimage of $\bot$;
\item if $h(m,x)$ does not appear among $h(m,x')$ for all $x'\ne x$, then $\omega(M(m,x))=\mu(x)$;
\item if $h(m,x)$ does not appear among $h(m,x')$ for $x'<x$, but appears among $h(m,x')$ for $x'>x$, then $\omega(M(m,x))\ne\bot$.
\end{itemize}
The first condition is computable, so we can start by checking it. Then we need to implement a choice between the second and third conditions. The third one defines an enumerable set of pairs $\langle m,x\rangle$, and we can use the property (3) for this set and the enumeration $\tau(\langle m,x\rangle)=\mu(x)$ obtained by using~(1).

The fixed point theorem provides some $m$ such that $[m](x)=M(m,x)$ for all $x$; in particular, $[m]$ is total. Let us prove that $x\mapsto h(m,x)$ is a bijection. Assume that $h(m,x)=h(m,x')$ for $x<x'$. Then the first item says that $\omega(M(m,x'))=\bot$, so $\omega(h(m,x'))=\omega([m](x'))=\omega(M(m,x'))=\bot$.  On the other hand, the pair $\langle m,x\rangle$ is served by the third item, so $\omega(M(m,x))\ne\bot$, and thus $\omega(h(m,x))=\omega([m](x))=\omega(M(m,x))\ne\bot$. We get a contradiction with the assumption $h(m,x)=h(m,x')$. Now we know that $x\mapsto h(m,x)$ is injective, and $\omega(M(m,x))=\mu(x)$ according to the second item, so $\omega(h(m,x))=\omega([m](x))=\omega(M(m,x))=\mu(x)$ for all $x$, and we get an injective reduction of $\mu$ to $\omega$.
\end{proof}

\begin{rem}
Knowing that effectively optimal numberings (as well as effectively optimal machines) are unique up to a length-bounded isomorphism, one can consider --- for a given function or for a given string --- its frequency among the first $N$ functions (or the first $N$ outputs of the machine), and then take $\limsup$ and $\liminf$ of these frequencies. Schnorr's result guarantees that these quantities are well defined (up to a constant factor). 

The $\limsup$ is in fact constant (it is easy to construct some effectively optimal numbering or effectively optimal machine with this property, so it is $\Omega(1)$ for every numbering or machine). The $\liminf$ for decompressors, as Muchnik has shown, equals $\m^{\mathbf{0}'}$, the relativized a priori probability (see~\cite[Problem 213]{VUS2013} and \cite{BienvenuS2012}). 

\emph{Question}: What can we say about limit frequencies for computable functions and/or $\Sigma_n$-sets? Do we get some kind of relativized complexity as well?\\
\end{rem}

\textbf{Acknowledgments}. This paper is based on the work done while D.D. was visiting LIRMM (Montpellier) and Poncelet laboratory (Moscow). We thank our colleagues from both laboratories (in particular the ESCAPE team, and the Kolmogorov seminar group) for hospitality. A.S. thanks Antti Valmari for interesting discussion (during RuFiDiM seminar in Turku) that was the starting point for some of the arguments in this paper. Finally we thank two anonymous referees for their very helpful feedback.

\nocite{CaludeD2015,CaludeS2008}

\bibliographystyle{plain}
\bibliography{BienvenuDesfontainesShen}

\begin{thebibliography}{10}

\bibitem{BienvenuDS2015}
Laurent Bienvenu, Damien Desfontaines, and Alexander Shen.
\newblock What percentage of programs halt?
\newblock In Magn\'{u}s~M. Halld\'{o}rsson, Kazuo Iwama, Naoki Kobayashi, and
  Bettina Speckmann, editors, {\em ICALP 2015, 42nd International Colloquim,
  Kyoto, Japan, July 6-10, 2015. LCNS 9134.}, volume 9134 of {\em Lecture Notes
  in Computer Science}, pages 219--230, Berlin, Heidelberg, 2015. Springer
  Berlin Heidelberg.

\bibitem{BMSV2012}
Laurent Bienvenu, Andrej Muchnik, Alexander Shen, and Nikolai Vereshchagin.
\newblock Limit complexities revisited [once more].
\newblock Technical report, \texttt{arxiv:1204.0201}, 2012.

\bibitem{BienvenuS2012}
Laurent Bienvenu and Alexander Shen.
\newblock Random semicomputable reals revisited.
\newblock In Michael~J. Dinneen, Bakhadyr Khoussainov, and Andr{\'e} Nies,
  editors, {\em Computation, Physics and Beyond}, volume 7160 of {\em Lecture
  Notes in Computer Science}, pages 31--45. Springer, 2012.

\bibitem{CaludeHKW1998}
Cristian Calude, Peter Hertling, Bakhadyr Khoussainov, and Yongge Wang.
\newblock Recursively enumerable reals and {C}haitin {O}mega numbers.
\newblock In {\em Symposium on Theoretical Aspects of Computer Science
  \textup(STACS 1998\textup)}, volume 1373 of {\em Lecture Notes in Computer
  Science}, pages 596--606. Springer, 1998.

\bibitem{CNSS}
Cristian Calude, Andr\'e Nies, Ludwig Staiger, and Frank Stephan.
\newblock Universal recursively enumerable sets of strings.
\newblock {\em Theoretical Computer Science}, 412(22):2253--2261, 2011.

\bibitem{Calude2006}
Cristian~S. Calude, Elena Calude, and Michael~J. Dinneen.
\newblock A new measure of the difficulty of problems.
\newblock Technical report, Center for Discrete Mathematics and Theoretical
  Computer Science, CDMTCS-277, 2006.

\bibitem{CaludeD2015}
Cristian~S. Calude and Damien Desfontaines.
\newblock Universality and almost decidability.
\newblock {\em Fundam. Inform.}, 138(1-2):77--84, 2015.

\bibitem{CaludeS2008}
Cristian~S. Calude and Michael Stay.
\newblock Most programs stop quickly or never halt.
\newblock {\em Advances in Applied Mathematics}, 40:295--308, 2008.

\bibitem{DowneyH2010}
Rodney Downey and Denis Hirschfeldt.
\newblock {\em Algorithmic randomness and complexity}.
\newblock Theory and Applications of Computability. Springer, 2010.

\bibitem{DowneyJockuschSchupp}
Rodney~G. Downey, Carl~G. Jockusch~Jr, and Paul~E. Schupp.
\newblock Asymptotic density and computably enumerable sets.
\newblock {\em Journal of Mathematical Logic}, 13(02), 2013.

\bibitem{HamkinsMiasnikov}
Joel~D. Hamkins and Alexei Miasnikov.
\newblock The halting problem is decidable on a set of asymptotic probability
  one.
\newblock {\em Notre Dame Journal of Formal Logic}, 47(4), 2006.

\bibitem{KohlerSchindelhauerZiegler}
Sven K{\"o}hler, Christian Schindelhauer, and Martin Ziegler.
\newblock On approximating real-world halting problems.
\newblock In {\em Fundamentals of computation theory}, pages 454--466.
  Springer, 2005.

\bibitem{KuceraS2001}
Antonin Ku{\v c}era and Ted Slaman.
\newblock Randomness and recursive enumerability.
\newblock {\em SIAM Journal on Computing}, 31:199--211, 2001.

\bibitem{LV2007}
Ming Li and Paul Vit\'anyi.
\newblock {\em An Introduction to Kolmogorov Complexity and Its Applications}.
\newblock Springer, 3 edition, 2007.

\bibitem{Lynch1974}
Nancy Lynch.
\newblock Approximations to the halting problem.
\newblock {\em Journal of Computer and System Sciences}, 9:143--150, 1974.

\bibitem{Miller2004}
Joseph~S. Miller.
\newblock Every {2}-random real is {K}olmogorov random.
\newblock {\em Journal of Symbolic Logic}, 69(3):907--913, 2004.

\bibitem{Nies2009}
Andr{\'e} Nies.
\newblock {\em Computability and randomness}.
\newblock Oxford Logic Guides. Oxford University Press, 2009.

\bibitem{Rogers}
Hartley Rogers.
\newblock {\em Theory of Recursive Functions and Effective Computability}.
\newblock The MIT Press, 1967.

\bibitem{CJ1999}
C.~Schindelhauer and A.~Jakoby.
\newblock The non-recursive power of erroneous computation.
\newblock In {\em Proc. 19th Foundations of Software Technology and Theoretical
  Computer Science, Lecture Notes in Computer Science 1738}, pages 394--406,
  1999.

\bibitem{Schnorr1974}
Claus~Peter Schnorr.
\newblock Optimal enumerations and optimal {G}\"odel numberings.
\newblock {\em Mathematical Systems Theory}, 8(2):181--191, 1974.

\bibitem{Valmari}
Antti Valmari.
\newblock The asymptotic proportion of hard instances of the halting problem.
\newblock Technical report, \texttt{arxiv:1307.7066v2}, November 2014.

\bibitem{VUS2013}
Nikolai Vereshchagin, Vladimir Uspensky, and Alexander Shen.
\newblock {\em Kolmogorov complexity and algorithmic randomness \textup{(In
  Russian. See \texttt{www.lirmm.fr/\~{}ashen} for the draft translation.)}}.
\newblock MCCME, 2013.

\end{thebibliography}

\end{document}